\documentclass[english]{amsart}
\usepackage{mathptmx}

\usepackage[T1]{fontenc}
\usepackage[latin1]{inputenc}
\usepackage{amsthm}
\usepackage{graphicx}
\usepackage{bbm}
\usepackage{mathptmx}
\usepackage{setspace}
\usepackage{color}
\usepackage{cancel}
\usepackage{a4wide}
\makeatletter

\usepackage{babel}

\numberwithin{equation}{section} 
\numberwithin{figure}{section} 
\theoremstyle{plain}
\newtheorem{thm}{Theorem}
  \theoremstyle{plain}
  \newtheorem{cor}{Corollary}[section]
  \theoremstyle{plain}
  \newtheorem{lem}[cor]{Lemma}
  \theoremstyle{plain}
    \newtheorem{prop}[cor]{Proposition}
    \theoremstyle{remark}
       \newtheorem{rem}[cor]{Remark}
	 \theoremstyle{plain}

\theoremstyle{plain}
  
  \newtheorem{defn}[thm]{Definition}
\usepackage{amssymb,amsmath}
\usepackage{bbm}

\renewcommand{\phi}{\varphi}

\renewcommand{\emptyset}{\varnothing}

\def\R{\mathbb{R}}

\def\N{\mathbb{N}}
\def\Z{\mathbb{Z}}

\def\a{\alpha}
\def\l{\lambda}

\def\b{\beta}

\def\om{\omega}
\SetSymbolFont{operators}{bold}{OT1}{cmr}{bx}{n}
\SetSymbolFont{letters}{bold}{OML}{cmm}{b}{it}
\SetSymbolFont{symbols}{bold}{OMS}{cmsy}{b}{n}

\newcommand{\var}{\varepsilon}
\newtheorem{example}{Example}

\usepackage{babel}

\numberwithin{equation}{section} 
\numberwithin{figure}{section} 

  \theoremstyle{plain}

  \theoremstyle{plain}

\makeatother

\usepackage{babel}
  \addto\captionsbritish{}
  \addto\captionsbritish{}
  \addto\captionsbritish{}
  \addto\captionsbritish{}
  \addto\captionsenglish{}
  \addto\captionsenglish{}
  \addto\captionsenglish{}
  \addto\captionsenglish{}

\begin{document}

\title[Hausdorff dimension spectrum of CGDMS]{The Hausdorff dimension spectrum of  Conformal Graph Directed Markov Systems and applications to Nearest Integer continued fractions}
\author{A. Ghenciu, S. Munday, M. Roy}

\begin{abstract}
In this paper, we consider two dynamical systems associated to the nearest integer continued fraction, and show that both of them have full Hausdorff dimension spectrum.
\end{abstract}

\maketitle

\section{Introduction and statement of results}

It is well known that every irrational real number $x$ can be written uniquely as an infinite fraction
\begin{eqnarray}x=a_0+\frac{1}{a_1 + \frac{1}{a_2+\frac{1}{a_3 + \cdots}}},\end{eqnarray} where $a_0\in \Z$ and each  $a_i\in \N$, for $i\geq1$. This is the regular continued fraction expansion of $x$. This classical and extremely well-studied expansion is far from being the only interesting one that has been introduced. Another class of expansions, a generalisation of the regular continued fraction (RCF), are the semi-regular continued fraction (SRCF) expansions. These are expansions which improve the approximation properties of the regular continued fraction (for more information on this see \cite{DK} and references therein) and they are defined as follows. A SRCF expansion is a finite or infinite fraction
\[
[b_0; \var_1b_1, \var_2b_2, \var_3b_3, \ldots]:=b_0+\frac{\var_1}{b_1+\frac{\var_2}{b_2+\frac{\var_3}{b_3+\ldots}}},
\]
with $\var_n=\pm1$, $b_0\in \Z$ and $b_n\in \N$ for all $n\geq1$, subject to the conditions that  $\var_{n+1}+b_n\geq1$ for all $n\geq1$, and, if the fraction is infinite, we have infinitely often that $\var_{n+1}+b_n\geq2$.

In this paper, we are interested in a particular example of a SCRF, namely, the nearest integer continued fraction (NICF). This expansion was introduced by Minnigerode in 1873 \cite{min73}, and has been studied quite intensively by several authors, starting with Hurwitz \cite{hur89}. The NICF is a SRCF satisfying $b_n\geq2$ and $b_n+\var_{n+1}\geq2$ for all $n\geq1$. The NICF is intimately related to the regular continued fraction, via the process of singularization, which we now describe (see \cite{Kra} for more details and further references). First, for any two positive integers $a$ and $b$, and $\xi \in (0, 1)$, observe that
\[
a+\frac{1}{1+\frac{1}{b+\xi}} = (a+1)+\frac{-1}{b+1+\xi}.
\]
Then, if we have a SRCF expansion
\begin{eqnarray}\label{star1}
[b_0; \var_1 b_1, \var_2 b_2, \var_3 b_3, \ldots]
\end{eqnarray}
with $b_{k+1}=\var_{k+1}=\var_{k+2} = 1$ for some $k\geq0$, we can replace (\ref{star1}) by
\begin{eqnarray}\label{star2}
[b_0; \var_1 b_1, \var_2 b_2, \ldots, \var_{k-1}b_{k-1}, \var_k (b_{k}+1), -(b_{k+2}+1), \var_{k+3}b_{k+3}, \ldots].
\end{eqnarray}
Now consider the RCF expansion of an irrational number $x$ and the following algorithm. Suppose that we have $a_{n+1}=\cdots = a_{n+m}=1$, for $m\in \N\cup\{\infty\}$, $n\geq0$, $a_{n+m+1}\neq1$ and $a_n\neq1$ (assuming $n>0$). Then singularize $a_{n+1}$, $a_{n+3}$, $a_{n+5}$, and so on, in turn. One immediately verifies that the expansion obtained in this way is the NICF expansion of $x$. Notice that this implies, in particular, that every irrational number admits an infinite NICF expansion. Moreover, this expansion is unique.

Let $[b_0; \var_1 b_1, \var_2 b_2, \ldots]$ be an infinite SRCF (we ignore finite expansions from here on, as they are only countably many). Then it is shown in \cite[Theorem 1.7]{Kra} that  there exist sequences $(p_n)_{n\geq-1}$ and $(q_n)_{n\geq-1}$ in $\Z$ that satisfy the recurrence relations
\[
\left\{
  \begin{array}{ll}
    p_{-1}:=1,\ \ \  p_0:=b_0 ,& \hbox{$p_n=b_np_{n-1}+\var_np_{n-1}$,} \\
    q_{-1}:=0,\ \ \  q_0:=1, & \hbox{$q_n=b_nq_{n-1}+\var_nq_{n-1}$.}
  \end{array}
\right.
\]
It is also shown that for all $n\geq-1$, we have $\mathrm{gcd}(p_n, q_n)=1$ and $\mathrm{gcd}(q_{n}, q_{n+1})=1$. Then, if for each $n\geq0$ we define $p_n/q_n:=[b_0; \var_1 b_1, \ldots, \var_n b_n]$,  the continued fraction  $[b_0; \var_1 b_1, \var_2 b_2, \ldots]$ is said to be convergent if and only if $\lim_{n\to\infty}p_n/q_n$ exists and is finite. It turns out that every SRCF converges to an irrational number (see \cite{Kra} again, and references therein), so it makes sense to refer to $(p_n/q_n)_{n\geq-1}$ as the sequence of convergents to the number $x= [b_0; \var_1 b_1, \var_2 b_2, \ldots]$. For the NICF example, we have that $|q_{n-1}|\leq |q_n|$ for all $n\geq1$ (see Corollary 1.9 in \cite{Kra}).

Much of the work done on the NICF has concentrated on its Diophantine approximation properties (see, for instance, \cite{jag85}, \cite{jagkra}, \cite{rock80}). We instead will focus on the question of its Hausdorff dimension spectrum, which we define shortly below. For this, it will be helpful to have a more  dynamical representation of the NICF.
Let $[\cdot]$ denote the integer part function. Then the nearest integer continued
fraction expansions are determined by the discontinuous transformation
$T:[-1/2,1/2]\to[-1/2,1/2]$ which is defined by setting
\[
T(x):=\left\{
        \begin{array}{ll}
          \frac{1}{x}-\left[\frac{1}{x}-\frac{1}{2}\right], & \hbox{if $x\neq0$;} \\
          0, & \hbox{if $x=0$.}
        \end{array}
      \right.
\]
By ``determined by'', we mean that the digits of the NICF can be found using the map $T$ as follows: For all $n\geq1$,
\[
b_n=b_n(x)=\left[\frac{1}{T^{n-1}(x)}\right].
\]
Note that the digits are now integers, instead of natural numbers coupled with a sign. That is, our $b_n$ generated by the map $T$ is equal to $\var_nb_n$ from above.
The inverse branches of $T$ are the conformal\footnote{Recall that a map is conformal is the derivative at every point is a rotation.} maps
\[
\varphi_b(x)=\frac{1}{b+x},\,\,\, |b|>1,
\]
whose respective domains are
\[
\varphi_2:\left[0,\frac{1}{2}\right]
          \to
          \left[\frac{2}{5},\frac{1}{2}\right]\subset\left[0,\frac{1}{2}\right],
\]
\[
\varphi_{-2}:\left[-\frac{1}{2},0\right]
             \to
             \left[-\frac{1}{2},-\frac{2}{5}\right]\subset\left[-\frac{1}{2},0\right],
\]
and
\[
\varphi_b:\left[-\frac{1}{2},\frac{1}{2}\right]
          \to
          \left[\frac{1}{b+1/2},\frac{1}{b-1/2}\right]
                       \subset
                       \left\{
                       \begin{array}{cll}
                       \left[0,\frac{1}{2}\right]  & \mbox{ if } & b>2 \\
                       \left[-\frac{1}{2},0\right]  & \mbox{ if } & b<-2
                       \end{array}
                       \right\}
                       \subset\left[-\frac{1}{2},\frac{1}{2}\right].
\]

Now,  let $E=\{b\in\Z:|b|\geq2\}$. Let $F\subset E$,
and let $J_F$ be the set of all numbers in $[-1/2,1/2]$ which can be represented by an infinite NICF with all digits belonging to the set $F$. If $F=E$, then
$J_E$ is the set of all irrational numbers in the interval $[-1/2,1/2]$. This set has Lebesgue
measure $1$. However, if $F$ is a proper subset of $E$, then the set $J_F$ has
Lebesgue measure $0$. Therefore, to distinguish between these sets, we use the Hausdorff dimension, which we will denote by $\dim_H(\cdot)$. (We will assume basic familiarity with properties of the Hausdorff dimension throughout, and refer to \cite{Fal}.)

The problem we are interested in is this: Given $0\leq t\leq1$, does there exist
a set $F\subset E$ such that $\dim_H(J_F)=t$? For the RCF expansion, this was an open problem for several
years, known as the Texan Conjecture. It was answered affirmatively for $0\leq t\leq1/2$
by Mauldin and Urba\'nski~\cite{mutr}. Later, it was answered positively for all $0\leq t\leq1$
by Kesseb\"ohmer and Zhu~\cite{KZ}. It is then said that the standard continued fraction expansion
has full Hausdorff dimension spectrum. Similar results were obtained by Ghenciu
for the backward continued fraction expansions~\cite{AGpara} and the Gauss-like continued
fraction expansions~\cite{G}. To solve these problems, these authors
associated to each continued fraction expansion an infinite conformal iterated function
system (cIFS), which, very briefly, is a finite or infinite set of conformal contracting similarities of a compact metric space.

In this paper, we consider questions related to the Hausdorff dimension spectrum of the NICF. The observant reader will have already spotted the main difficulty  - the NICF cannot be associated to an IFS, since the domains of the inverse branches of the map $T$ are not all the same space. To get around this problem, we need to introduce graph directed Markov systems. Then, there are two natural IFSs that can be associated to the NICF. The first is the IFS obtained by restricting the digits of the NICF to the set $F:=\{b\in\Z: |b|\geq3\}$, which we shall denote by $\Phi_F$.  The second is an IFS associated to one of the vertices of the graph directed Markov system we will use to describe the NICF; we will denote this IFS by $\Phi^{(v)}$, but for the details of how it is defined  we defer to Section 5. Our main results concern the dimension spectra of these two systems.

\begin{thm}\label{mainthm1}
$\Phi_F$ has full Hausdorff dimension spectrum.
\end{thm}

\begin{thm}\label{mainthm2}
$\Phi^{(v)}$ has full Hausdorff dimension spectrum.
\end{thm}

The paper is organised as follows. In Section \ref{sec_prelim}, we will introduce much of the preliminary material needed for the rest of the paper, beginning with the definition of a conformal graph directed Markov system. Section 3 contains a collection of lemmas needed for the proof of Theorem \ref{mainthm1}; the proof itself can be found in Section 4. Section 5 contains the details necessary to construct an IFS associated to the vertex of a GDMS and the proof of Theorem \ref{mainthm2}. Finally, we add an appendix containing further background results on CGDMSs, mostly these results are given simply to clear up small inaccuracies in previously available proofs.

\section{Preliminaries}\label{sec_prelim}
\subsection{Graph directed Markov systems}
Let us first introduce graph directed Markov systems. To do this,  we need a directed multigraph
$(V,E,i,t)$ and an associated incidence matrix $A$, i.e., a matrix containing only 0s and 1s. The multigraph consists
of a finite set $V$ of vertices, a (possibly infinitely) countable set $E$
of directed edges and two functions $i,t:E\to V$, where $i(e)$ is
the initial vertex of edge $e$ and $t(e)$ is its terminal vertex. The incidence matrix  $A$ of size $\#E\times \# E$ indicates which edge(s) may follow any given edge. In other words, $A_{ef}=1$  if and only if $t(e)=i(f)$. For later use, let us also introduce some more notation. The set $E_A^\infty$ of one-sided infinite $A$-admissible words is defined to be
\[
E_A^\infty:=\left\{\omega=\omega_1\omega_2\ldots\in
E^\infty:A_{\omega_i\omega_{i+1}}=1,\ \forall i\geq1\right\}.
\]
The set of all finite subwords of $E_A^\infty$ will be denoted by
$E_A^*$. The \emph{length} of any word $\omega$ is defined to be the number of letters it is made up of, and will be denoted by $|\omega|$. For each $n\geq1$, the set of all subwords of $E_A^\infty$ of length $n$ shall be
denoted by $E^n_A$. There is a unique word of length $0$ in $E_A^*$ called
the \emph{empty word}.
If $\omega \in E_A^\infty$ and $n\geq1$, then we write $\omega|_n$ for the initial $n$-block of the word $\omega$, that is,
\[
\omega|_n=\omega_1\omega_2\ldots\omega_n.
\]
A \emph{Graph Directed Markov System} (GDMS) consists of a directed multigraph $(V,E,i,t)$,
an incidence matrix $A$, a set of non-empty compact metric spaces $\{X_v\}_{v\in V}$
and a set of $1$-to-$1$ contractions $\{\varphi_e:X_{t(e)}\to X_{i(e)}\}_{e\in E}$
with Lipschitz constant $s$, where $0<s<1$. Sometimes, in a slight abuse of notation, we will refer to this set of contractions as a GDMS, but only when the context is clear. The matrix $A$ tells us which contractions can be applied after each other, in the following way.
For each $\omega\in E_A^*$,
the map coded by $\omega$ is defined to be
\[
\varphi_\omega:=\varphi_{\omega_1}\circ\ldots\circ\varphi_{\omega_{|\omega|}}
              :X_{t(\omega)}\to X_{i(\omega)},
\]
where $t(\omega):=t(\omega_{|\omega|})$ and $i(\omega):=i(\omega_1)$.

\begin{rem}\label{defn_IFS}
If the set of vertices in the GDMS is a singleton and all the entries in the incidence matrix are 1, then the GDMS is an {\em iterated function system}, abbreviated to IFS. More concretely, an IFS is a countable set of contraction maps with Lipschitz constant $0<s<1$ which map a compact metric space into itself. Iterated function systems were well-studied before GDMSs were introduced, particularly in terms of generating fractal sets (see \cite{Fal}).
\end{rem}

Returning to our GDMS, for each $\omega\in E_A^\infty$, the sets
$\{\varphi_{\omega|_n}(X_{t(\omega_n)})\}_{n\geq1}$ form a
decreasing sequence of non-empty compact subsets of $X_{i(\omega_1)}$. Also, since for every $n\geq1$ we have that
\[
\mbox{diam}(\varphi_{\omega|_n}(X_{t(\omega_n)})) \leq s^n\mbox{diam}(X_{t(\omega_n)})\leq s^n\max\{\mbox{diam}(X_v):v\in V\},
\]
the intersection
\[
\bigcap_{n\geq1}\varphi_{\omega|_n}\left(X_{t(\omega_n)}\right)
\]
is a singleton whose element is denoted by $\pi(\omega)$. If we set $X$ to be the disjoint union of the sets $\{X_v\}_{v\in V}$, then the map
\[
\pi:E_A^\infty\to X
\]
defined in this way is called the \emph{coding map}. The set
\[
J:=J_{E,A}=\pi(E_A^\infty)
\]
is called the \emph{limit set} of the GDMS $S$.

From this point on in the paper, we make two simplifying assumptions
about the directed graph. First, we assume that for all $e\in E$
there exists $f\in E$ so that $A_{ef}=1$.
Otherwise, if there were $e\in E$ so that $A_{ef}=0$ for every $f\in E$,
then the limit set $J_{E,A}$ would be the same as the limit set
$J_{E\setminus\{e\},A}$ (in the construction of this latter set, $A$ is
restricted to $(E\setminus\{e\})^2$). Second, we assume that for every
vertex $v\in V$ there exists $e\in E$ so that $i(e)=v$. Otherwise, if
there existed $v\in V$ such that no edge has for initial vertex $v$, then
the limit set $J$ would be the same if the vertex set were $V\setminus\{v\}$.

We emphasize that we have two directed graphs that play an important
role in our study. The first one is the given multigraph $(V,E,i,t)$.
The second one, $G_{E,A}$, 
is determined by the matrix $A$. The vertices of $G_{E,A}$ are the edges of the
first one, and $G_{E,A}$ has a directed edge from $e$ to $f$ if and only if
$A_{ef}=1$. Therefore $G_{E,A}$ has infinitely many vertices and edges if and
only if $E$ is an infinite set.

We will also need the following properties of the incidence matrix $A$. Firstly, $A$ is said to be \emph{irreducible} if for any
two edges $e,f\in E$ there exists a word $\omega\in E_A^*$ so that
$e\omega f\in E_A^*$.
This is equivalent to saying that the directed graph $G_{E,A}$ is
\emph{strongly connected}, i.e. for any two vertices there exists
a path starting from one and ending at the other. The matrix $A$ is said to be \emph{finitely irreducible} if there exists
a finite set $\Omega\subset E_A^*$ so that for any two edges $e,f\in E$
there is a word $\omega\in\Omega$ so that $e\omega f\in E_A^*$.

The matrix $A$ is called \emph{primitive} if there exists $p\geq1$ such that
all the entries of $A^p$ are positive (written $A^p>0$) or, in other words,
for any two edges $e,f\in E$ there exists a word $\omega\in E_A^{p-1}$ so that
$e\omega f\in E_A^{p+1}$. Similarly, the matrix $A$ is called \emph{finitely primitive} if there exist $p\geq1$
and a finite set $\Omega\subset E_A^{p-1}$
such that for any two edges $e,f\in E$ there is a word $\omega\in\Omega$
so that $e\omega f\in E_A^{p+1}$.
\\
\\
A GDMS is called \emph{conformal}, and hence a CGDMS, if the following conditions
are satisfied:
\begin{enumerate}
\item[(1)] For every $v\in V$, the set $X_v$ is a compact connected subset of a
Euclidean space $\mathbf{R}^d$ (the dimension $d$ common for all vertices)
and $X_v=\overline{\mbox{Int}(X_v)}$.
\item[(2)] (Open Set Condition (OSC))
For every $e,f\in E$, $e\neq f$,
\[
\varphi_e\left(\mbox{Int}(X_{t(e)})\right)\bigcap\varphi_f\left(\mbox{Int}(X_{t(f)})\right)=\emptyset.
\]
\item[(3)] For every vertex $v\in V$ there exists an open connected set $W_v\supset X_v$
so that for every $e\in E$ with $t(e)=v$, the map $\varphi_e$ extends to a $C^1$ conformal
diffeomorphism of $W_v$ into $W_{i(e)}$.
\item[(4)] (Cone property) There exists $\gamma,l>0$, 
such that for every $x\in X$ there exists an open cone
$\mbox{Con}(x,\gamma,l)\subset\mbox{Int}(X)$ with vertex $x$, central
angle of measure $\gamma$, and altitude $l$.
\item[(5)] There are two constants $L\geq1$ and $\alpha>0$ so that
\[
\left||\varphi_e'(y)|-|\varphi_e'(x)|\right|\leq L\|(\varphi_e')^{-1}\|^{-1}
\|y-x\|^\alpha
\]
for every $e\in E$ and for every pair of points $x,y\in W_{t(e)}$,
where $|\varphi_e'(x)|$ represents the norm of the derivative of $\varphi_e$
at $x$. This says that the norms of the derivative maps are all H\"older continuous functions of
order $\alpha$ with H\"older constant depending on the map.
\end{enumerate}

\begin{rem}
As explained in~\cite{gdms}, condition~(5) plays a central role in
dimension $d=1$. If $d\geq2$ and we are given a GDMS
which satisfies conditions~(1) and~(3),
then it automatically fulfills condition~(5)
with $\alpha=1$. In this paper, we will only be considering GSMSs in dimension $d=1$. This also means that the property (4) will not concern us, as it is always satisfied for $d=1$.
\end{rem}
As a straightforward consequence of~(5), we obtain the famous Bounded Distortion Property (BDP):
\begin{enumerate}
\item[(6)] There exists $K\geq1$ such that
for all $\omega\in E_A^*$ and for all $x,y\in W_{t(\omega)}$,
\begin{equation}\label{BDP}
|\varphi_\omega'(y)|\leq K|\varphi_\omega'(x)|.
\end{equation}
\end{enumerate}

\subsection{GDMS for the  NICF}

As mentioned already in the introduction, the NICF cannot be described by an iterated function system, as the inverse branches are not all defined upon the same domain. Let us recall the definition of these branches:
\[
\varphi_b(x)=\frac{1}{b+x},\,\,\, |b|\geq2,
\]
where $\phi_2$ is defined on the interval $[0, 1/2]$, $\phi_{-2}$ on $[-1/2, 0]$ and $\phi_b$, for $|b|\geq3$, is defined upon $[-1/2, 1/2]$.
Therefore the composition of these inverse branches
are subject to some restrictions.  We shall
describe the restrictions by means of an incidence matrix $A$ and by
identifying the branch $\varphi_b$ with the letter $b$. Thus, the composition
$\varphi_{e}\circ\varphi_f$ shall be allowed if and only if $A_{ef}=1$,
that is,  if and only if the word $ef$ is $A$-admissible.
Let $E=\{b\in\Z:|b|\geq2\}$ and $A:E^2\to\{0,1\}$ be
the matrix defined by setting
\[
A_{ef}:=\left\{\begin{array}{llcll}
              1 & \mbox{ if } & |e|>2 &              & \\
              1 & \mbox{ if } & e=2   & \mbox{ and } & f>0 \\
              0 & \mbox{ if } & e=2   & \mbox{ and } & f<0 \\
              1 & \mbox{ if } & e=-2  & \mbox{ and } & f<0 \\
              0 & \mbox{ if } & e=-2  & \mbox{ and } & f>0.
              \end{array}
       \right.
\]

We introduce an infinite conformal graph directed Markov system which reflects
the backward trajectories of $T$, that is, the composition of the inverse branches
$\{\varphi_b\}_{|b|\geq2}$ of $T$. As these inverse branches have three different domains, we shall need three
vertices. Let the set of vertices and attached spaces be $V=\{v,w,z\}$ and
\[
X_v=\left[-\frac{1}{2},\frac{1}{2}\right],
\hspace{1cm}
X_w=\left[0,\frac{1}{2}\right]
\hspace{1cm} \mbox{ and } \hspace{1cm}
X_z=\left[-\frac{1}{2},0\right].
\]
Note that the alphabet $E$ is not sufficient to construct a graph
directed Markov system. In \cite{gdms}, page 1, the authors state: "the incidence matrix  $A$ determines which edge(s) may follow any given edge. In other words, if $A_{ef}=1$  then $t(e)=i(f)$" . We need copies of some of the letters in $E$.

\begin{itemize}
\item Draw a graph with the three vertices $v$, $w$ and $z$;
\item Draw a self-loop based at vertex $v$ for each $|e|>2$;
\item Draw an edge from vertex $v$ to vertex $w$ and identify it by the letter $2$.
\item Draw an edge from $w$ to $v$ for each $e>2$ and identify it by $\overline{e}$.
      (These edges are identified by $\overline{e}$ to distinguish them from the self-loops $e$.
      However, their corresponding generators $\varphi_{\overline{e}}$ have for codomain $X_w$,
      whereas the generators $\varphi_e$ corresponding to the self-loops have for codomain $X_v$.
      Hence, they are maps given by the same expression, having the same domain but different
      codomains.);
\item Draw a self-loop based at $w$ and identify it as $\overline{2}$;
\item Draw an edge from vertex $v$ to vertex $z$ and identify it by $-2$;
\item Draw an edge from $z$ to $v$ for each $e<-2$ and identify it by $\overline{e}$.
      (These edges are identified by $\overline{e}$ to differentiate them from the self-loops $e$.
      Note that their corresponding generators $\varphi_{\overline{e}}$ have for codomain $X_z$,
      whereas the generators corresponding to the self-loops $\varphi_e$
      have for domain $X_v$. Thus, they are maps given by the same expression,
      with the same domain but different codomains.);
\item Draw a self-loop based at $z$ and identify it by $\overline{-2}$.
\end{itemize}
We hence obtain a graph directed system $\Phi$. Define a matrix $\overline{A}$ that exactly
reflects that graph. This means that the new alphabet is
$\overline{E}=\{e:|e|\geq 2\}\cup\{\overline{e}:|\overline{e}|\geq 2\}$.
Observe that the matrix $\overline{A}$ contains essentially the same information as
the original matrix $A$. As mentioned earlier, the generators of this system are
\[
\varphi_e(x)=\varphi_{\overline{e}}(x)=\frac{1}{e+x}
\]
with domains and codomains reflecting the above graph.

Let $\om\in E_{\overline{A}}^*$. Then
\[
\varphi_\om(x)=\frac{p_{|\om|}+xp_{|\om|-1}}{q_{|\om|}+xq_{|\om|-1}},
\]
where the $p_n=p_n(\om)$'s and $q_n=q_n(\om)$'s are as defined in the introduction. Therefore,
\begin{eqnarray}\label{derivative}
|\varphi_\om'(x)|=\frac{1}{(q_{|\om|}+xq_{|\om|-1})^2}
\end{eqnarray}
since
\[
p_{n-1}q_n-q_{n-1}p_n=(-1)^n
\]
for all $1\leq n\leq|\om|$.

\section{Lemmas for later}

In this section, we  give a series of Lemmas that will be used in the proof of Theorem \ref{mainthm1}. So, let us recall that for this theorem we are using the alphabet $F:=\{b\in \Z: |b|\geq3\}$ and the full shift space associated to $F$. We will also make extensive use of the recurrence relations for the NICF which were given in the introduction. Here, though, we are using them directly on the symbolic alphabet. Note that this really means we are taking a word consisting of letters from $F$, applying the inverse coding map to it, then calculating the $q_n$s for the NICF. To save complicated notation, we will simply write it directly for the letters of $\omega$.  Remember that these recurrence relations are given, for $1\leq n\leq|\omega|$, by
\begin{equation}
\label{p equation}
p_n = \omega_n p_{n-1} + p_{n-2}.
\end{equation}
\begin{equation}
\label{q equation}
q_n = \omega_n q_{n-1} + q_{n-2}.
\end{equation}
where: $q_0=1$ and $q_1=\omega_1.$\\

We begin with a series of estimates on the size of the denominators of the convergents.

\begin{lem}\label{2.1}Let $\alpha = \frac{3 - \sqrt{5}}{2}$. Then for every $n \geq 2$ we have that
$$\frac{|q_{n-1}|}{|q_n|} \leq \alpha.$$
\begin{proof}
Using \ref{q equation}, for every $n \geq 2$, we obtain
$$\frac{q_n}{q_{n-1}}=\omega_n + \frac{q_{n-2}}{q_{n-1}}.$$
Thus, inductively, we have that
$$\left|\frac{q_{n-1}}{q_n}\right|=\frac{1}{\left|\omega_n + \frac{q_{n-2}}{q_{n-1}}\right|} \leq \frac{1}{|\omega_n| - \left|\frac{q_{n-2}}{q_{n-1}}\right|} \leq \frac{1}{3-\alpha}=\alpha.$$
\end{proof}
\end{lem}
From this, we obtain the following immediate corollary:
\begin{cor}\label{2.2}
For every $n \geq 2$, we have:
$$\frac{|q_{n-1}|}{|q_n|} \leq \frac{1}{|\omega_n|-\alpha}.$$
\end{cor}

The next lemma provides an estimate in the other direction.
\begin{lem}\label{2.3}
For every $n \geq 2$, we have:
$$\frac{|q_{n-1}|}{|q_n|} \geq \frac{1}{|\omega_n|+\alpha}.$$

\begin{proof}
Using  (\ref{q equation}) as in Lemma \ref{2.1}, we first obtain that
\begin{eqnarray}\label{formerLem2.4}
\left|\frac{q_{n-1}}{q_n}\right|=\frac{1}{\left|\omega_n + \frac{q_{n-2}}{q_{n-1}}\right|} \geq \frac{1}{|\omega_n| + \left|\frac{q_{n-2}}{q_{n-1}}\right|}.
\end{eqnarray}
At this point we use Lemma \ref{2.1} to conclude that:
$$\frac{|q_{n-1}|}{|q_n|} \leq \frac{1}{|\omega_n|+\alpha}.$$
\end{proof}
\end{lem}

Finally,  we fix $\omega\in F^{n+1}$ and study the behaviour of the function $G_\om:\left[\frac{-1}{2},\frac{1}{2}\right] \to \left(-\infty,\infty\right)$, where we define $G_\om(x):=\left|\frac{q_n+xq_{n-1}}{q_{n+1} + x q_n}\right|$. 
This function will be used in the following section.

\begin{lem}\label{2.4}
For each $n\in \N$, we have
$$G_\om(x) \leq \frac{3}{2} \frac{1}{|\omega_{n+1}|-\alpha}.$$
\begin{proof}
First we make the simple observation that
$$G_\om(x) \leq \frac{|q_n| + |x| |q_{n-1}|}{|q_{n+1}| - |x||q_n|} \leq \frac{|q_n| + \frac{1}{2} |q_{n-1}|}{|q_{n+1}| - \frac{1}{2}|q_n|}.$$
Then, using first Lemma \ref{2.1} and then Corollary \ref{2.2} we obtain the desired estimate:
$$G_\om(x) \leq \frac{|q_n| + \frac{1}{2}\frac{2}{5} |q_n|}{|q_{n+1}| - \frac{1}{2}\frac{2}{5}|q_{n+1}|} \leq \frac{\frac{6}{5}|q_n|}{\frac{4}{5}|q_{n+1}|} \leq \frac{3}{2} \frac{1}{|\omega_{n+1}| - \alpha}.$$
\end{proof}
\end{lem}

\begin{lem}\label{2.5}
For each $n\in \N$, we have
$$G_\om(x) \geq \frac{2}{3} \frac{1}{|\omega_{n+1}|+\alpha}.$$
\begin{proof}
Similarly to the proof of Lemma \ref{2.4}, we first notice that
$$G_\om(x) \geq \frac{|q_n| - |x| |q_{n-1}|}{|q_{n+1}| + |x||q_n|} \geq \frac{|q_n| - \frac{1}{2} |q_{n-1}|}{|q_{n+1}| + \frac{1}{2}|q_n|}.$$
Thus, in light of Lemmas \ref{2.1} and \ref{2.3},
$$G_\om(x) \geq \frac{|q_n| - \frac{1}{2}\frac{2}{5} |q_n|}{|q_{n+1}| + \frac{1}{2}\frac{2}{5}|q_{n+1}|} \geq \frac{\frac{4}{5}|q_n|}{\frac{6}{5}|q_{n+1}|} \geq \frac{2}{3} \frac{1}{|\omega_{n+1}| + \alpha}.$$
\end{proof}
\end{lem}

\section{Proof of Theorem \ref{mainthm1}}

In order to prove our first main theorem, we will need several results which originate (in slightly different form) in \cite{AGthesis} and \cite{KZ}. To state them, we must make two further definitions: The topological pressure of $P_F$ of the IFS $\Phi_F$ is defined for each $t\in \R$ by
\[
P_F(t):= \lim_{n\to\infty} \frac{1}{n}\log (Z_n), 
\]
where $Z_n:= \sum_{\om\in F^n} ||\phi_\om'||^t$. Also, for any subset $G\subseteq F$, we write $\lambda_G:=\exp (P_G)$. 

\begin{thm}\label{Me1}
Let $\Phi$ be a conformal iterated function system. Let $F\subset E$ and $e\in E$. If
$M_e>0$ is such that
\[
\|\varphi_{\om e\overline{\om}}'\|\leq M_e\|\varphi_{\om\overline{\om}}'\|
\]
for all words $\om\in F^*$ and $\overline{\om}\in(F\cup\{e\})^*$, then
\[
\l_F(t)\leq\l_{F\cup\{e\}}(t)\leq\l_F(t)+M_e^t.
\]
\end{thm}


\begin{thm}\label{me1}
Let $\Phi$ be a conformal iterated function system. Let $F\subset E$ and $e\in E$. If
$m_e>0$ is such that
\[
\|\varphi_{\om e\overline{\om}}'\|\geq m_e\|\varphi_{\om\overline{\om}}'\|
\]
for all words $\om\in F^*$ and $\overline{\om}\in(F\cup\{e\})^*$, then
\[
\l_{F\cup\{e\}}(t)\geq\l_F(t)+m_e^t.
\]
\end{thm}


\begin{rem}\label{distort}
Under the hypothesis of Theorems~\ref{Me1} and ~\ref{me1}, the existence of $M_e$ and $m_e$ is
guaranteed by the bounded distortion property of the system. In particular,
$M_e$ can be taken to be $K\|\varphi_e'\|$, whereas $m_e$ can be taken to
be $K^{-1}\inf_{x\in X}|\varphi_e'(x)|$ or $K^{-2}\|\varphi_e'\|$, where $K$ is the constant appearing in (\ref{BDP}).
\end{rem}


The following theorem is a weakening of Theorem~2.2 in~\cite{KZ}.
\begin{thm}\label{fullspec}
Let $\Phi_\N$ be a conformal iterated function system indexed by the natural
numbers $\N$. For every $b\in\N$, let
$L_b=\{b+1,b+2,\ldots\}$. If for every
$b\in\N$, every $F\subset\{1,2,\ldots,b-1\}$ and $0<t\leq\dim_H(J_{\Phi_\N})$
we have
\begin{equation}\label{full}
\l_{F\cup\{b\}}(t)\leq\l_{F\cup L_b}(t),
\end{equation}
then $\Phi_\N$ has full Hausdorff dimension spectrum.
\end{thm}

\begin{thm}\label{Mme}
Let $\Phi_\N$ be a conformal iterated function system
such that $\dim_H(J_{\Phi_\N})\leq1$.
For every $b\in\N$, let $S_b=\{1,2,\ldots,b-1\}$ and $L_b=\{b+1,b+2,\ldots\}$.
If for some $b\in\N$ there are positive constants $M_b$ and $\{m_c\}_{c>b}$ such that
\[
m_b\|\varphi_{\om\overline{\om}}'\|
\leq
\|\varphi_{\om b\overline{\om}}'\|
\leq
M_b\|\varphi_{\om\overline{\om}}'\|
\]
for all words $\om\in S_b^*$ and $\overline{\om}\in(S_b\cup\{b\})^*$ 
and so that
\[
M_b<\sum_{c=b+1}^\infty m_c,
\]
then
\[
\l_{F\cup\{b\}}(t)\leq\l_{F\cup L_b}(t)
\]
for all $F\subset\{1,2,\ldots,b-1\}$ and all $0<t\leq\dim_H(J_{\Phi_\N})$, so (due to Theorem 7), $\Phi_\N$ has full Hausdorff dimension spectrum.
\end{thm}

\begin{proof}
Fix $b\in\N$ as in the statement of the theorem.
Pick any $F\subset\{1,2,\ldots,b-1\}$ and $0\leq t\leq\dim_H(J_{\Phi_\N})$.
Choose $n\geq b+1$ so that $M_b\leq\sum_{c=b+1}^n m_c$.
Using Theorems~\ref{Me1} and~\ref{me1} (the second repeatedly), we obtain that
\begin{eqnarray*}
\l_{F\cup\{b\}}(t)
\leq\l_F(t)+M_b^t
\leq\l_F(t)+\left(\sum_{c=b+1}^n m_c\right)^t
&\leq&\l_F(t)+\sum_{c=b+1}^n m_c^t \\
&\leq&\l_{F\cup\{b+1\}}(t)+\sum_{c=b+2}^n m_c^t \\
&\leq&\l_{F\cup\{b+1,b+2\}}(t)+\sum_{c=b+3}^n m_c^t \\
&\leq&\ldots \\
&\leq&\l_{F\cup\{b+1,b+2,\ldots,n\}}(t) \\
&\leq&\l_{F\cup L_b}(t).
\end{eqnarray*}
\end{proof}


\begin{rem}\label{useG}
In light of Lemmas \ref{2.4} and \ref{2.5}, we can easily obtain some better constants $M_b$ and $m_b$ than were given in Remark \ref{distort}. To see this, let $\om$ and $\overline{\om}$ be two admissible words and let $b$ be a letter from our alphabet. We have the following:
\[
\varphi'_{\om b \overline{\om}}(x) = \varphi'_{\om b}(\varphi_{\overline{\om}}(x))\varphi'_{\overline{\om}}(x)\leq \left(\frac{3}{2} \frac{1}{|b|-\alpha}\right)^2
\varphi'_{\om }(\varphi_{\overline{\om}}(x)) \varphi'_{\overline{\om}}(x) \leq \left(\frac{3}{2} \frac{1}{|b|-\alpha}\right)^2
\varphi'_{\om {\overline{\om}}}(x)
\]
Thus we can take $M_b = \left(\frac{3}{2} \frac{1}{|b|-\alpha}\right)^2$ and, similarly, $m_b = \left(\frac{2}{3} \frac{1}{|b|+\alpha}\right)^2$.
\end{rem}

Now we are ready to prove our first main result, with the aid of the following lemma.

\begin{lem}\label{2.6}
For every $k \geq 4$, we have:
\begin{equation}\label{*lem2.6}
\left(\frac{9}{4}\right) \frac{1}{(k - \alpha)^2} \leq 2  \left(\frac{4}{9}\right) \sum_{j \geq k+1} \frac{1}{(j+\alpha)^2}
\end{equation}
\end{lem}
\begin{proof}
Using the Integral Test yields that
$$\sum_{j \geq k+1} \frac{1}{(j+\alpha)^2} \geq \frac{1}{k+1+\alpha}.$$
On the other hand, for every $k \geq 4$ we have that
$$\frac{9}{4}\frac{1}{(k - \alpha)^2} \leq \frac{8}{9} \frac{1}{k+1+\alpha}.$$
This finishes the proof.
\end{proof}

Now, combining Lemma \ref{2.6} with Remark \ref{useG} completes the proof of Theorem \ref{mainthm1}.

\section{Proof of Theorem \ref{mainthm2}}\label{sectionIFS2}

We will shortly describe in detail the IFS associated to the vertex $v$ of the GDMS introduced in Section 2 for the NICF. We refer back to that section for the definition of the alphabet $\overline{E}$. First, we give the general construction.

Suppose we have a CGDMS $\Phi= (V, E, i, t, A, \{X_v\}_{v\in V}, \{\phi_e\}_{e\in E})$. For every vertex $v\in V$ we define the alphabet $E_v\subset E_A^*$
by induction as the union $\cup_{n=1}^\infty E_{v,n}$ as follows. To begin, define
\[ E_{v,1}:=\{e\in E:i(e)=t(e)=v\}. \]
Suppose now that all the sets $E_{v,k}\subset E_A^k$, for $k=1,\ldots,n$, have
been defined. We then say that $\om\in E_A^{n+1}$ belongs to $E_{v,n+1}$
if $i(\om)=t(\om)=v$ and $\om$ is not the concatenation of words from
$\cup_{k=1}^n E_{v,k}$. In other words, $E_{v,n}$ is the set of all
$A$-admissible first-return loops of length $n$ originating from the vertex $v$.
By construction, no element of $E_v$ is a concatenation
of other elements of $E_v$.
We further define the matrix $A^{(v)}:E_v\times E_v\to\{0,1\}$ by
$A^{(v)}_{\om\overline{\om}}=1$ if and only if $\om\overline{\om}\in E_A^*$,
that is, if and only if $A_{\om_{|\om|}\overline{\om}_1}=1$, where
$\om,\overline{\om}\in E_v$.
The system  $\Phi_v$ is a CIFS whenever $\Phi$  satisfies $A_{ef}=1$  if and only if $t(e)=i(f)$.
We then say that $\Phi_v$ is the CIFS associated with the CGDMS $\Phi$ via vertex $v$.

For our example, the associated iterated
function system based at vertex $v$, which we will call $\Phi^{(v)}$, has for alphabet the first-return loops based at
$v$, i.e., all the loops $|e|>2$,
all the $(n+1)$-loops $2(\overline{2})^{n-1}\overline{e}$ with $e>2$,
and all the $(n+1)$-loops $(-2)(\overline{-2})^{n-1}\overline{f}$ with $f<-2$,
where $n\geq1$.
Thus,
\[
\overline{E}^{(v)}=\left\{e:|e|>2\right\}
                   \bigcup
                   \left\{2(\overline{2})^{n-1}\overline{e}:e>2,n\geq1\right\}
                   \bigcup
                   \left\{(-2)(\overline{-2})^{n-1}\overline{f}:f<-2,n\geq1\right\}.
\]
Recall from \cite{gdms} that the Hausdorff dimension of the limit set $J_{\overline{E}^{(v)}}$ of the associated
iterated function system $\Phi^{(v)}$ is equal to the Hausdorff dimension of the limit set $J$ of
the original graph directed system $\Phi$.

To shorten the notation, we shall replace $\overline{e}$ by $e$ whenever it is clear from
the context which of $e$ and/or $\overline{e}$ is meant. For instance, in order to respect the
graph, the word $(2)^n3$ is really the word $2(\overline{2})^{n-1}\overline{3}$.

Next, we impose the following order on the alphabet $\overline{E}^{(v)}$:
\[
\begin{array}{l}
 -3, 3, -4, 4, \\
 (-2)(-3), (-2)^2(-3), (-2)^3(-3), (2)(3), (2)^2(3), (2)^3(3), \\
 (-2)(-4), (-2)^2(-4), (-2)^3(-4), (-2)^4(-4), (-2)^4(-3), \\
 (2)(4), (2)^2(4), (2)^3(4), (2)^4(4), (2)^4(3), \\
 -5, 5, \\
 (-2)(-5), (-2)^2(-5), (-2)^3(-5), (-2)^4(-5), (-2)^5(-5), (-2)^5(-4), (-2)^5(-3), \\
 (2)(5), (2)^2(5), (2)^3(5), (2)^4(5), (2)^5(5), (2)^5(4), (2)^5(3), \\
 -6, 6, \\
 (-2)(-6), (-2)^2(-6), (-2)^3(-6), (-2)^4(-6), (-2)^5(-6), (-2)^6(-6), (-2)^6(-5), (-2)^6(-4),  \\
 (-2)^6(-3), (2)(6), (2)^2(6), (2)^3(6), (2)^4(6), (2)^5(6), (2)^6(6), (2)^6(5), (2)^6(4), (2)^6(3), \ldots \\
\end{array}
\]

For the calculations to follow, we will also need the following slightly different estimate of distortion.

\begin{lem}\label{distortword}
For any $\om,\tau\in E_A^*$ we have
\[
K_\om^{-1}\inf_{y\in X}|\varphi_\tau'(y)|\cdot\|\varphi_{\om\overline{\om}}'\|
\leq\|\varphi_{\om\tau\overline{\om}}'\|
\leq K_\om\|\varphi_\tau'\|\cdot\|\varphi_{\om\overline{\om}}'\|
\]
where $K_\om$ is a constant of distortion for $\varphi_\om$, i.e.,
\[
K_\om:=\sup_{x,y\in X}\frac{|\varphi_\om'(x)|}{|\varphi_\om'(y)|}\leq K,
\]
where
\[
K=\sup_{\om\in E_A^*}K_\om
\]
is a constant of distortion for the entire system.
\end{lem}

\begin{proof}
Fix $\om,\tau\in E_A^*$. Then
\begin{eqnarray*}
\|\varphi_{\om\tau\overline{\om}}'\|
&=&\sup_{x\in X}|\varphi_{\om\tau\overline{\om}}'(x)| \\
&=&\sup_{x\in X}\left(|\varphi_\om'(\varphi_{\tau\overline{\om}}(x))|\cdot
                      |\varphi_\tau'(\varphi_{\overline{\om}}(x))|\cdot
                      |\varphi_{\overline{\om}}'(x)|\right) \\
&\leq&\|\varphi_\tau'\|
   \sup_{x\in X}\left(|\varphi_\om'(\varphi_{\tau\overline{\om}}(x))|\cdot
                      |\varphi_{\overline{\om}}'(x)|\right) \\
&=&\|\varphi_\tau'\|
   \sup_{x\in X}\left(|\varphi_\om'(\varphi_{\overline{\om}}(x))|\cdot
                      |\varphi_{\overline{\om}}'(x)|\cdot
                      \frac{|\varphi_\om'(\varphi_{\tau\overline{\om}}(x))|}
                           {|\varphi_\om'(\varphi_{\overline{\om}}(x))|}
                      \right) \\
&\leq&\|\varphi_\tau'\|K_\om
   \sup_{x\in X}\left(|\varphi_\om'(\varphi_{\overline{\om}}(x))|\cdot
                      |\varphi_{\overline{\om}}'(x)|\right) \\
&=&K_\om\|\varphi_\tau'\|
   \sup_{x\in X}|\varphi_{\om\overline{\om}}'(x)| \\
&\leq&K_\om\|\varphi_\tau'\|\cdot\|\varphi_{\om\overline{\om}}'\|.
\end{eqnarray*}
On the other hand,
\begin{eqnarray*}
\|\varphi_{\om\tau\overline{\om}}'\|
&=&\sup_{x\in X}|\varphi_{\om\tau\overline{\om}}'(x)| \\
&=&\sup_{x\in X}\left(|\varphi_\om'(\varphi_{\tau\overline{\om}}(x))|\cdot
                      |\varphi_\tau'(\varphi_{\overline{\om}}(x))|\cdot
                      |\varphi_{\overline{\om}}'(x)|\right) \\
&\geq&\inf_{y\in X}|\varphi_\tau'(y)|\cdot
   \sup_{x\in X}\left(|\varphi_\om'(\varphi_{\tau\overline{\om}}(x))|\cdot
                      |\varphi_{\overline{\om}}'(x)|\right) \\
&=&\inf_{y\in X}|\varphi_\tau'(y)|\cdot
   \sup_{x\in X}\left(|\varphi_\om'(\varphi_{\overline{\om}}(x))|\cdot
                      |\varphi_{\overline{\om}}'(x)|\cdot
                      \frac{|\varphi_\om'(\varphi_{\tau\overline{\om}}(x))|}
                           {|\varphi_\om'(\varphi_{\overline{\om}}(x))|}
                      \right) \\
&\geq&\inf_{y\in X}|\varphi_\tau'(y)|\cdot K_\om^{-1}
   \sup_{x\in X}\left(|\varphi_\om'(\varphi_{\overline{\om}}(x))|\cdot
                      |\varphi_{\overline{\om}}'(x)|\right) \\
&=&K_\om^{-1}\inf_{y\in X}|\varphi_\tau'(y)|\cdot
   \sup_{x\in X}|\varphi_{\om\overline{\om}}'(x)| \\
&\geq&K_\om^{-1}\inf_{y\in X}|\varphi_\tau'(y)|\cdot\|\varphi_{\om\overline{\om}}'\| \\
&\geq&K_\om^{-2}\|\varphi_\tau'\|\cdot\|\varphi_{\om\overline{\om}}'\|.
\end{eqnarray*}
\end{proof}

In order to apply this result, we will need the following.

\begin{lem}\label{lem_2s}
Let $\om\in \overline{E}_{\overline{A}}^*$ and suppose that $\om$ has the form $\om=\om_1\ldots \om_{n-k-1} \underbrace{2\ldots 2}_{k \ \text{times}} \om_n$, where $|\om_{n-k-1}|$ and $|\om_n|$ are both at least equal to 3. Then,
\[
\left|\frac{q_{n-1}}{q_n}\right| \leq \frac{k+2}{2k+5}.
\]
\end{lem}

\begin{proof}
We will use the estimate from the proof of Lemma \ref{2.1} repeatedly:
\begin{eqnarray*}
\left|\frac{q_{n-1}}{q_n}\right|&\leq& \frac{1}{3- \left|\frac{q_{n-2}}{q_{n-1}}\right|}\\
&\leq& \frac{1}{3s- \frac{1}{2-\left|\frac{q_{n-3}}{q_{n-2}}\right|}}\\
&\leq& \ldots\\
&\leq& \frac{1}{3-\frac{1}{2- \frac{1}{2 - \ddots  \frac{1}{2-\left|\frac{q_{n-k-2}}{q_{n-k-1}}\right| }}}}.
\end{eqnarray*}
Now, since $|\om_{n-k-1}|\geq3$, the last ratio satisfies $\left|\frac{q_{n-k-2}}{q_{n-k-1}}\right|\leq \frac12$. A simple calculation then finishes the proof. (Note that if the word $\om = 2\ldots 2\om_{k+1}$, the calculation stops one step earlier and we obtain a slightly better estimate.)
 \end{proof}

\begin{lem}\label{lem_est_distort}
For the system $\Phi^{(v)}$, we may take $K=25/9$.
\end{lem}

\begin{proof}
Let $\om \in (\overline{E}^{(v)})_{\overline{A}}^*$ and $x,y\in X_v$. Recalling that $|\om|$ refers to the length of the word $\om$ considered as consisting of letters from $\overline{E}$, we note from the proof of Lemma \ref{2.1} that
\[
\left|\frac{q_{k-1}}{q_{k}}\right|\leq \frac{1}{|\om_{|\om|}|-\left|\frac{q_{k-2}}{q_{k-1}}\right|}
\]
Note that $|\om_{|\om|}|\geq3$, by the definition of the alphabet $\overline{E}^{(v)}$. It then follows from the fact that $|q_{n-1}|\leq |q_n|$ for all $n\geq1$ and from the above observation that
\begin{eqnarray}\label{*p17}
\left|\frac{q_{|\omega|-1}}{q_{|\omega|}}\right|\leq \frac{1}{3-1}=\frac12.
\end{eqnarray}
Using (\ref{derivative}) and (\ref{*p17}), we obtain that
\begin{eqnarray*}
\frac{|\varphi_\om'(x)|}{|\varphi_\om'(y)|}
&=&\left|\frac{q_{|\om|}+yq_{|\om|-1}}{q_{|\om|}+xq_{|\om|-1}}\right|^2
\leq\left(\frac{|q_{|\om|}|+|y||q_{|\om|-1}|}{|q_{|\om|}|-|x||q_{|\om|-1}|}\right)^2 \\
&\leq&\left(\frac{|q_{|\om|}|+\frac{1}{2}|q_{|\om|-1}|}{|q_{|\om|}|-\frac{1}{2}|q_{|\om|-1}|}\right)^2 \leq\left(\frac{|q_{|\om|}|+\frac{1}{2}\cdot\frac{1}{2}|q_{|\om|}|}
                 {|q_{|\om|}|-\frac{1}{2}\cdot\frac{1}{2}|q_{|\om|}|}\right)^2 \\
&=&25/9.
\end{eqnarray*}
Hence $25/9$ is a constant of bounded distortion for our system.
\end{proof}

We are now ready to begin the proof of Theorem \ref{mainthm2}.

\begin{proof}[Proof of Theorem \ref{mainthm2}]
We will mostly use Theorem~\ref{Mme} to establish that the associated iterated
function system $\Phi^{(v)}$ has full spectrum, and the proof is split into several different cases.
Note that we can express the
letters in $\overline{E}^{(v)}$ in the following general form: $2^jk$ and $(-2)^j(-k)$,
where $j\geq0$ and $k>2$. The calculations below involve the derivatives of the generators. Due to the symmetry in the system, we have
$\|\varphi_b'\|=\|\varphi_{-b}'\|$ for all $|b|\geq2$. Consequently, the letters $2^jk$ and
$(-2)^j(-k)$ can be treated in the same manner. Without loss of generality, we will restrict our attention to
the letters $2^jk$.

{\bf Case of the letters $2^jk$, where $j>k$.}

According to our ordering of the letters of $\overline{E}^{(v)}$,
if $j>k\geq3$ then the letter $2^jk$ precedes the letters $\pm l$, where $l\geq j+1$.
It is thus sufficient to prove that
\begin{eqnarray}\label{*p18}
M_{2^jk}\leq2\sum_{l=j+1}^\infty m_l,
\end{eqnarray}
where, according to Remark~\ref{distort}, we may take
\[
m_l=K^{-1}\inf_{x\in X_v}|\varphi_l'(x)|
=\frac{9}{25}\frac{1}{(l+\frac{1}{2})^2},
\]
and
\[
M_{2^jk}
=K\|\varphi_{2^jk}'\|
\leq K\|\varphi_2'\|^j\|\varphi_k'\|
=\frac{25}{9}\left(\frac{1}{2^2}\right)^j\frac{1}{(k-\frac{1}{2})^2}
\leq\frac{25}{9}\left(\frac{1}{2^2}\right)^j\frac{1}{(3-\frac{1}{2})^2}
=\frac{4}{9}\left(\frac{1}{2^2}\right)^j.
\]
Substituting these values into (\ref{*p18}), we see that it suffices to prove that
\[
\frac{4}{9}\left(\frac{1}{2^2}\right)^j
\leq2\cdot\frac{9}{25}\sum_{l=j+1}^\infty \frac{1}{(l+\frac{1}{2})^2},
\]
or, in other words, that
\[
\frac{50}{81}\leq 2^{2j+1}\sum_{l=j+1}^\infty \frac{1}{(l+\frac{1}{2})^2}.
\]
Using the integral test yields that
\[
\sum_{l=j+1}^\infty \frac{1}{(l+\frac{1}{2})^2}
\geq\frac{1}{j+1+\frac{1}{2}}.
\]
Consequently, proving (\ref{*p18}) boils down to proving that
\[
\frac{50}{81}(j+1+\frac{1}{2})\leq 2^{2j+1},
\]
which is certainly satisfied for $j>k\geq3$.

{\bf Case of the letters $2^jk$, where $1\leq j\leq k$.}

According to our ordering of the letters of $\overline{E}^{(v)}$,
if $1\leq j\leq k$ then the letter $2^jk$ precedes the letters $\pm l$, where $l\geq k+2$.
It is thus sufficient to prove that
\[
M_{2^jk}\leq2\sum_{l=k+2}^\infty m_l,
\]
where, as above, we may take
\[
m_l=\frac{9}{25}\frac{1}{(l+\frac{1}{2})^2}\ \text{ and } \ M_{2^jk}\leq \frac{25}{9}  \left(\frac{1}{2^2}\right)^j\frac{1}{(k-\frac{1}{2})^2}.
\]
Using the integral test again, we have that
\[
\sum_{l=k+2}^\infty \frac{1}{(l+\frac{1}{2})^2}
\geq\frac{1}{k+2+\frac{1}{2}}.
\]
Consequently, exactly analogously to the first case, it is sufficient to show that
\[
\left(\frac{25}{9}\right)^2\frac{(k+\frac{5}{2})}{(k-\frac{1}{2})^2}\leq 2^{2j+1}.
\]
It is then easy to show that the left-hand side is a decreasing
function of $k$ when $k\geq3$. Therefore it suffices to show that
\[
\left(\frac{25}{9}\right)^2\frac{(3+\frac{5}{2})}{(3-\frac{1}{2})^2}\leq 2^{2j+1}.
\]
One immediately verifies that this is true for all $j\geq1$.

{\bf Case of the letters $k$, where $k\geq6$.}

According to our ordering of the letters of $\overline{E}^{(v)}$,
the letter $k$, for $k\geq6$, precedes the letters $\pm l$, where $l\geq k+1$.
It is thus sufficient to prove that
\[
M_k\leq2\sum_{l=k+1}^\infty m_l,
\]
where once again using Remark~\ref{distort} we may take
\[
m_l=\frac{9}{25}\frac{1}{(l+\frac{1}{2})^2} \ \text{ and } \
M_k =K\|\varphi_k'\ |=\frac{25}{9}\frac{1}{(k-\frac{1}{2})^2}.
\]
It therefore suffices to prove that
\[
\left(\frac{25}{9}\right)^2\frac{1}{(k-\frac{1}{2})^2}
\leq 2\sum_{l=k+1}^\infty \frac{1}{(l+\frac{1}{2})^2}.
\]
Using the integral test again gives
\[
\sum_{l=k+1}^\infty \frac{1}{(l+\frac{1}{2})^2}
\geq\frac{1}{k+1+\frac{1}{2}}.
\]
Consequently, it is sufficient to show that
\begin{equation}\label{esti}
\left(\frac{25}{9}\right)^2\frac{(k+\frac{3}{2})}{(k-\frac{1}{2})^2}\leq 2.
\end{equation}
One immediately verifies that the left-hand side is a decreasing
function of $k$ when $k\geq3$. The smallest value of $k$ for which
relation~(\ref{esti}) holds is $k=6$.

{\bf Case of the letters $\pm5$.}

We have just proved the case $k\geq6$. To prove the result for smaller
values of $k$, we need better estimates on the distortion
and to consider more of the letters following $k$. Since the words $\om$
in Theorems~\ref{Me1} and~\ref{me1} can be taken to be composed of letters that precede $k$,
according to Lemma~\ref{distortword} we may always replace $K$ by $\max_\om K_\om$,
where the maximum is taken over
all words comprising only letters that precede $k$. Moreover, according to our
ordering of the letters of $\overline{E}^{(v)}$,
the letter $k$ precedes the letters $\pm l$, where $l\geq k+1$,
as well as the letters $2^rl$ and $(-2)^r(-l)$ for all $r\geq1$ and $l\geq k+1$.
It is thus sufficient to prove that
\begin{eqnarray}\label{*p19}
M_k
\leq2\sum_{l=k+1}^\infty m_l+2\sum_{l=k+1}^\infty\sum_{r=1}^\infty m_{2^rl}
=2\sum_{l=k+1}^\infty\left[m_l+\sum_{r=1}^\infty m_{2^rl}\right].
\end{eqnarray}
Using Remark~\ref{distort} once again, we may take
\[
m_l=\frac{9}{25}\frac{1}{(l+\frac{1}{2})^2} \ \text{ and }\ m_{2^rl}=K^{-1}\inf_{x\in X_v}|\varphi_{2^rl}'(x)|
\geq\frac{9}{25}\frac{1}{\left[(\frac{3}{2}+\sqrt{2})(1+\sqrt{2})^{r-1}\right]^2}
    \frac{1}{(l+\frac{1}{2})^2}.
\]
The latter inequality above comes from the following calculation: First observe that for the letter $2^r$ a straightforward induction argument shows that
$1\leq q_n\leq(1+\sqrt{2})^n$ for all $0\leq n\leq r$. Then
\begin{eqnarray*}
\inf_{x\in X_v}|\varphi_{2^r}'(x)|
&=&\inf_{x\in[0,1/2]}\frac{1}{(q_r+xq_{r-1})^2} \\
&=&\frac{1}{(q_r+\frac{1}{2}q_{r-1})^2} \\
&\geq&\frac{1}{\left[(1+\sqrt{2})^r+\frac{1}{2}(1+\sqrt{2})^{r-1}\right]^2} \\
&\geq&\frac{1}{\left[(\frac{3}{2}+\sqrt{2})(1+\sqrt{2})^{r-1}\right]^2}.
\end{eqnarray*}

For the left-hand side of (\ref{*p19}), according to Lemma~\ref{distortword}, we may choose
\[
M_k
=\left(\sup_{\om\prec5}K_\om\right)\|\varphi_k'\|
=\left(\frac{27}{17}\right)^2\frac{1}{(k-\frac{1}{2})^2},
\]
where the supremum is taken over all words $\om\in(\overline{E}^{(v)})_{\overline{A}}^*$ comprising
only letters that precede $5$.
Moreover, in light of Lemma \ref{lem_2s}, we have that $\sup_{\om\prec5}K_\om=\left(\frac{8}{5}\right)^2$ (this follows from a calculation identical to that in Lemma \ref{lem_est_distort}).
It therefore suffices to prove that
\[
\left(\frac{8}{5}\right)^2\frac{1}{(k-\frac{1}{2})^2}
\leq
2\cdot\frac{9}{25}\sum_{l=k+1}^\infty \frac{1}{(l+\frac{1}{2})^2}
\left\{1+\sum_{r=1}^\infty\frac{1}{\left[(\frac{3}{2}+\sqrt{2})(1+\sqrt{2})^{r-1}\right]^2}\right\},
\]
Using the integral test, it is sufficient to show that
\begin{equation}\label{ineq}
\frac{25}{18}\left(\frac{8}{5}\right)^2\frac{(k+\frac{3}{2})}{(k-\frac{1}{2})^2}
\leq
1+\frac{1}{(\frac{3}{2}+\sqrt{2})^2}\sum_{r=1}^\infty\frac{1}{[(1+\sqrt{2})^2]^{r-1}}
\end{equation}
It is again a straightforward calculation to show that the left-hand side is a decreasing
function of $k$ when $k\geq3$. It is then easy to establish that
\[
\frac{25}{18}\left(\frac{8}{5}\right)^2\frac{(5+\frac{3}{2})}{(5-\frac{1}{2})^2}
\leq1+\frac{1+\sqrt{2}}{2(\frac{3}{2}+\sqrt{2})^2}. \]
Relation~(\ref{ineq}) thus holds
for all $k\geq5$.

{\bf Case of the letters $\pm4$.}

We have so far proved the case $k\geq5$. To prove the result for smaller
values of $k$, we need an even better estimate on the distortion
and to take all the letters following $k$. Since the words $\om$
in Theorems~\ref{Me1} to~\ref{me1} can be taken to be composed of letters that precede $k$,
in light of Lemma~\ref{distortword} we may always replace $K$ by $\max_\om K_\om$,
where the maximum is taken over all words comprising only letters that precede $k$.
Moreover, according to our
ordering of the letters of $\overline{E}^{(v)}$,
the letter $k$ precedes the letters $\pm l$, where $l\geq k+1$,
as well as the letters $2^rl$ and $(-2)^r(-l)$ for all $r\geq1$ and $l\geq3$.
It is thus sufficient to prove that
\[
M_4\leq2\sum_{l=5}^\infty m_l+2\sum_{l=3}^\infty\sum_{r=1}^\infty m_{2^rl},
\]
where, according to Lemma~\ref{distortword}, we may take
\[
m_l=(\sup_{\om\prec l}K_\om)^{-1}\inf_{x\in X_v}|\varphi_l'(x)|
=\left(\frac{3l+5}{5l+7}\right)^2\frac{1}{(l+\frac{1}{2})^2},
\]
where the supremum is taken over all words $\om\in(\overline{E}^{(v)})_{\overline{A}}^*$ comprising
only letters that precede $l$
and
\[
m_{2^rl}=K^{-1}\inf_{x\in X_v}|\varphi_{2^rl}'(x)|
\geq\frac{9}{25}\frac{1}{\left[(\frac{3}{2}+\sqrt{2})(1+\sqrt{2})^{r-1}\right]^2}\frac{1}{(l+\frac{1}{2})^2}.
\]
Here we obtain that $\sup_{\om\prec l}K_\om= ((5l+7)/(3l+5))^2$ by applying Lemma \ref{lem_2s} and making a calculation as in the previous case.

We also have that
\[
M_4
=\left(\sup_{\om\prec4}K_\om\right)\|\varphi_4'\|
=\left(\frac{1+\frac{3-\sqrt{5}}{4}}{1-\frac{3-\sqrt{5}}{4}}\right)^2\frac{1}{(4-\frac{1}{2})^2},
\]
where the supremum is taken over all words $\om\in(\overline{E}^{(v)})_{\overline{A}}^*$ comprising
only letters that precede $4$. Indeed,
\[
\sup_{\om\prec4}K_\om
=\left(\frac{1+\frac{3-\sqrt{5}}{4}}{1-\frac{3-\sqrt{5}}{4}}\right)^2 = \left(\frac{7-\sqrt 5}{1+\sqrt 5}\right)^2,
\]
as one can show that for any word $\om\in(\overline{E}^{(v)})_{\overline{A}}^*$ comprising
only letters that precede $4$, we have
$|q_{n-1}|\leq\frac{3-\sqrt{5}}{2}|q_n|$ for all $0\leq n\leq|\om|$
since in this case we have $3$ repeated any finite number of times, and the solution in $[-1/2, 1/2]$ to the equation $x=1/(3-x)$ is  $\frac{3-\sqrt{5}}{2}$.
Then we again make a calculation as in Lemma \ref{lem_est_distort}.

It therefore suffices to prove that
\[
\left(\frac{7-\sqrt5}{1+\sqrt5}\right)^2
\left(\frac{4}{49}\right)^2
\leq
2\sum_{l=5}^\infty \left(\frac{3l+5}{5l+7}\right)^2\frac{1}{(l+\frac{1}{2})^2}
+2\cdot\frac{9}{25}\sum_{l=3}^\infty \sum_{r=1}^\infty
\frac{1}{\left[(\frac{3}{2}+\sqrt{2})(1+\sqrt{2})^{r-1}\right]^2}\frac{1}{(l+\frac{1}{2})^2},
\]
i.e.
\[
\left(\frac{7-\sqrt5}{1+\sqrt5}\right)^2
\left(\frac{4}{49}\right)^2
\leq
2\sum_{l=5}^\infty \left(\frac{3l+5}{5l+7}\right)^2\frac{1}{(l+\frac{1}{2})^2}
+\frac{18/25}{(\frac{3}{2}+\sqrt{2})^2}\sum_{r=1}^\infty\frac{1}{[(1+\sqrt{2})^2]^{r-1}}
 \sum_{l=3}^\infty \frac{1}{(l+\frac{1}{2})^2}.
\]
Using the integral test, we have that
\[
\sum_{l=k}^\infty \frac{1}{(l+\frac{1}{2})^2}
\geq\frac{1}{k+\frac{1}{2}}.
\]
Consequently, it is sufficient to show that
\[
\left(\frac{7-\sqrt5}{1+\sqrt5}\right)^2
\left(\frac{4}{49}\right)^2
\leq
2\sum_{l=5}^\infty \left(\frac{3l+5}{5l+7}\right)^2\frac{1}{(l+\frac{1}{2})^2}
 +\frac{18/25}{(\frac{3}{2}+\sqrt{2})^2}\frac{1}{1-\frac{1}{(1+\sqrt{2})^2}}
 \frac{1}{3+\frac{1}{2}}.
\]
Hence it is sufficient to show that
\[
\left(\frac{7-\sqrt5}{1+\sqrt5}\right)^2
\left(\frac{4}{49}\right)^2
\leq2\sum_{l=5}^\infty \left(\frac{3l+2}{5l+2}\right)^2\frac{1}{(l+\frac{1}{2})^2}
+\frac{18}{175}\frac{1+\sqrt{2}}{(\frac{3}{2}+\sqrt{2})^2}.
\]
Numerical calculations using Mathematica show that this relation is true.

{\bf Case of the letter $-3$.}

Rather than using Theorem~\ref{Mme}, we shall show directly that relation~(\ref{full}) holds.
Since $-3$ is the first letter in the alphabet $\overline{E}^{(v)}$, relation~(\ref{full})
holds as $-3$ is followed by $3$ and
\[
\l_{\{-3\}}(t)=\l_{\{3\}}(t)\leq\l_{\overline{E}^{(v)}\backslash\{-3\}}(t)
\]
for all $t\geq0$.

{\bf Case of the letter $3$.}

Again, we shall show directly that relation~(\ref{full}) holds.
Since $3$ is the second letter in the alphabet $\overline{E}^{(v)}$, relation~(\ref{full})
holds as
\[
\l_{\{-3\}}(t)\leq\l_{\{-3,3\}}(t)
\leq\l_{\{l:|l|\geq4\}}(t)
\leq\l_{\overline{E}^{(v)}\backslash\{-3,3\}}(t)
\leq\l_{\overline{E}^{(v)}\backslash\{-3\}}(t)
\]
for all $0\leq t\leq1$. Indeed, let us prove that
$\l_{\{-3,3\}}(t)\leq\l_{\{l:|l|\geq4\}}(t)$ for all $0\leq t\leq1$.
On the one hand, we have that
\begin{equation}\label{z3}
\l_{\{-3,3\}}(t)
\leq Z_{1,\{-3,3\}}(t)
=\|\varphi_{-3}'\|^t+\|\varphi_3'\|^t
=\frac{2}{(3-\frac{1}{2})^{2t}}
=\frac{2}{\left(\frac{5}{2}\right)^{2t}}.
\end{equation}
On the other hand, we have that
\[
\l_{\{l:|l|\geq4\}}(t)\geq K_4^{-1}Z_{1,\{l:|l|\geq4\}}(t)
=2K_4^{-1}\sum_{l=4}^\infty\|\varphi_l'\|^t
=2K_4^{-1}\sum_{l=4}^\infty\frac{1}{(l-\frac{1}{2})^{2t}},
\]
where $K_4$ is a constant of bounded distortion for the subsystem $\{\varphi_l:|l|\geq4\}$.
Thus, since $\l_{\{-3,3\}}(t)$ is finite and  $\l_{\{l:|l|\geq4\}}(t)$ is infinite whenever $t\leq1/2$, we have $\l_{\{-3,3\}}(t)<\l_{\{l:|l|\geq4\}}(t)$. When $t>1/2$,
it follows from the integral test that
\begin{equation}\label{z4}
\l_{\{l:|l|\geq4\}}(t)
\geq2K_4^{-1}\frac{1}{(2t-1)(4-\frac{1}{2})^{2t-1}}
=2K_4^{-1}\frac{1}{(2t-1)(\frac{7}{2})^{2t-1}}.
\end{equation}
One can show that for any word $\om\in\{l:|l|\geq4\}^*$, we have
$|q_{n-1}|\leq(2-\sqrt{3})|q_n|$ for all $0\leq n\leq|\om|$
as the worst case scenario is to have the letters $-4$ and/or $4$ repeated any
finite number of times, and $2-\sqrt3$ is the solution of  $x=1(4-x)$ in $[-1/2, 1/2]$. Then calculating as in Lemma \ref{lem_est_distort} again, we can take
\begin{equation}\label{k4}
K_4=\left(\frac{1+\frac{2-\sqrt{3}}{2}}{1-\frac{2-\sqrt{3}}{2}}\right)^2= \left(\frac{4-\sqrt3}{1+\sqrt3}\right)^2.
\end{equation}
According to~(\ref{z3}),~(\ref{z4}) and~(\ref{k4}), to prove
$\l_{\{-3,3\}}(t)\leq\l_{\{l:|l|\geq4\}}(t)$ when $t>1/2$ it suffices
to show that
\[
\frac{1}{(2t-1)}\left(\frac{5}{7}\right)^{2t}
\geq\frac{2}{7}\left(\frac{4-\sqrt3}{1+\sqrt3}\right)^2.
\]
By looking at its first derivative, it is easy to see that the left-hand side is a
decreasing function of $t$ on $1/2<t\leq1$. Thus, we only need to show that
\[
\left(\frac{5}{7}\right)^2
\geq\frac{2}{7}\left(\frac{4-\sqrt3}{1+\sqrt3}\right)^2.
\]
Numerical calculations show that this relation is true.

We have hence demonstrated that relation~(\ref{full}) holds for all letters of
$\overline{E}^{(v)}$ under the ordering we chose. Therefore the system
$\Phi^{(v)}$ has full spectrum
according to Theorem~\ref{fullspec}.
\end{proof}

\section{Appendix}

In this appendix we follow the ideas from \cite{gdms}, and add some explanatory examples.

The following is a restatement of Proposition~4.7.2 in~\cite{gdms} with an
annotated proof.

\begin{prop}\label{jivjv}
Suppose that $\Phi$ is a CGDMS with an irreducible matrix.
For every vertex $v\in V$ the limit set $J_{E_v}:=\pi(E_v^\infty)$ of $\Phi_v$
is contained in the subset $J_v:=\pi(\{\om\in E_A^\infty:i(\om)=v\})$ of the
limit set of $\Phi$. Moreover, $\overline{J_{E_v}}=\overline{J_v}$.
\end{prop}

\begin{proof}
Since $E_v^\infty\subset\{\om\in E_A^\infty:i(\om)=v\}$, we have
$J_{E_v}\subset J_v$. Hence $\overline{J_{E_v}}\subset\overline{J_v}$.
In order to prove the opposite inclusion it suffices to demonstrate that each
element of $J_v$ is the limit of elements of $J_{E_v}$. Indeed, let $x=\pi(\om)$,
where $\om\in E_A^\infty$ with $i(\om)=v$. Since $A$ is irreducible, for every
$n\in \N$
there exist $\a^{(n)}\in E_A^*$ and
$\b^{(n)}\in
E_v^\infty$ such that
$\om|_n\a^{(n)}\b^{(n)}\in E_A^\infty$. Since $\b^{(n)}\in E_v^\infty$ and
$i(\om)=v$, we have $\om|_n\a^{(n)}\b^{(n)}\in E_v^\infty$. Hence
$\pi(\om|_n\a^{(n)}\b^{(n)})\in J_{E_v}$ for every $n\in\N$ and thus
$\lim_{n\to\infty}\pi(\om|_n\a^{(n)}\b^{(n)})
=\pi(\lim_{n\to\infty}\om|_n\a^{(n)}\b^{(n)})=\pi(\om)=x$.
Consequently, $x\in\overline{J_{E_v}}$. Since $x$ was chosen arbitrarily in
$J_v$, we deduce that $J_v\subset\overline{J_{E_v}}$. Hence
$\overline{J_v}\subset\overline{J_{E_v}}$.
\end{proof}

We shall now compare the pressures of the original and the associated systems.

\begin{thm}\label{pressass}
If $\Phi$ is a CGDMS with a finitely irreducible matrix, then
\[
P(t)\leq\min_{v\in V}P_{E_v}(t) \mbox{ whenever } P(t)>0
\]
and 
\[
P(t)=\max_{v\in V}P_{E_v}(t) \mbox{ whenever } P(t)\leq0,
\]
where $P_{E_v}(t)$ is the pressure of the system $\Phi_v$.
\end{thm}

\begin{proof}
First, we prove that $P(t)\leq\min_{v\in V}\max\{P_{E_v}(t),0\}$. Fix $v\in V$.
If $P_{E_v}(t)=\infty$, then clearly $P(t)\leq\max\{P_{E_v}(t),0\}$. If
$P_{E_v}(t)<\infty$, then let $u>\max\{P_{E_v}(t),0\}$. Let $W\subset E_A^*$ be a finite set
witnessing the irreducibility of $A$. Let
$\rho:=\min\{\|\varphi_\tau'\|:\tau\in W\}$ and $\lambda:=\max\{|\tau|:\tau\in W\}$.
Then $\rho>0$ and $\lambda<\infty$. For every $e\in E$ let $\a(e),\b(e)\in W$ be
such that $i(\a(e))=v$, $t(\b(e))=v$ and $\a(e)e\b(e)\in E_A^*$. Set
$\a(\om):=\a(\om_1)$ and $\b(\om):=\b(\om_{|\om|})$ for every $\om\in E_A^*$.
Observe that the function $\om\mapsto\a(\om)\om\b(\om)$ is at most $\lambda$-to-one.
Indeed, suppose that $\om,\tau\in E_A^*$ are such that
$\a(\om)\om\b(\om)=\a(\tau)\tau\b(\tau)$. If $|\a(\om)|=|\a(\tau)|$, then
$\a(\om)\om\b(\om)=\a(\tau)\tau\b(\tau)$ forces $\a(\om)=\a(\tau)$.
This in turn imposes that $\om\b(\om)=\tau\b(\tau)$. Without loss of generality,
we may assume that $|\om|\leq|\tau|$. Then
$\tau=\om\star\b(\om)|_{|\tau|-|\om|}$. There are at most
$\l$ such $\tau$ since $\b(\cdot)\in W\subset\bigcup_{k=1}^\lambda E_A^k$.
As
$\a(\cdot)\in W\subset\bigcup_{k=1}^\lambda E_A^k$, there are at most $\l$ of
the $\a(\cdot)$'s that are of different lengths, and for each of these there are at most
$\l$ preimages. Thus, the function
$\om\mapsto\a(\om)\om\b(\om)$ is at most $\lambda^2$-to-one.
Furthermore, notice that $\a(\om)\om\b(\om)\in\bigcup_{k=|\om|}^{2\lambda+|\om|}E_A^k$.
Moreover, recall that $|\om|_v\leq|\om|$ for every $\om\in E_v^*$. Then
\begin{eqnarray*}
\sum_{\om\in E_A^*}\|\varphi_\om'\|^t e^{-u|\om|}
&\leq&(K\rho^{-1})^{2t}\sum_{\om\in E_A^*}
      \|\varphi_{\a(\om)\om\b(\om)}'\|^t e^{-u|\om|} \\
&\leq&(K\rho^{-1})^{2t}e^{2\lambda u}\sum_{\om\in E_A^*}
      \|\varphi_{\a(\om)\om\b(\om)}'\|^t e^{-u|\a(\om)\om\b(\om)|} \\
&\leq&\lambda^2(K\rho^{-1})^{2t}e^{2\lambda u}\sum_{\tau\in E_v^*}\|\varphi_\tau'\|^t e^{-u|\tau|_v} \\
&<&\infty.
\end{eqnarray*}
The first inequality is a direct repercussion of the bounded distortion of the system.
The second inequality follows from the fact that $|\a(\om)\om\b(\om)|\leq2\lambda+|\om|$
for every $\om\in E_A^*$ and that $u>0$. The third inequality is a consequence
of the fact the function $\om\mapsto\a(\om)\om\b(\om)$ is at most $\lambda^2$-to-one,
that $|\tau|_v\leq|\tau|$ for every $\tau\in E_v^*$, and that $u>0$. Finally, the last
inequality follows from Theorem~2.1.3 in~\cite{gdms} since $u>P_{E_v}(t)$.
Since $\sum_{\om\in E_A^*}\|\varphi_\om'\|^t e^{-u|\om|}<\infty$, Theorem~2.1.3
in~\cite{gdms} affirms that $u>P(t)$. Since this is true for every
$u>\max\{P_{E_v}(t),0\}$, we deduce that $\max\{P_{E_v}(t),0\}\geq P(t)$.
Since this holds for every $v\in V$, we conclude that
$P(t)\leq\min_{v\in V}\max\{P_{E_v}(t),0\}$.

In particular, note that if $P(t)>0$, then $P(t)\leq\min_{v\in V}P_{E_v}(t)$.

Secondly, we show that $\max_{v\in V}P_{E_v}(t)\leq P(t)$ whenever $P(t)\leq0$.
Let $v\in V$. Let $t$ be such that $P(t)<0$ and $P(t)<u\leq0$. Then
$\sum_{\om\in E_A^*}\|\varphi_\om'\|^t e^{-u|\om|}<\infty$  according
to~Theorem~2.1.3 in~\cite{gdms} since $u>P(t)$.
Since $u\leq0$, we deduce that
\[
\sum_{\tau\in E_v^*}\|\varphi_\tau'\|^t e^{-u|\tau|_v}
\leq\sum_{\om\in E_A^*}\|\varphi_\om'\|^t e^{-u|\om|}<\infty.
\]
Thus, $P_{E_v}(t)\leq u$ by Theorem~2.1.3 in~\cite{gdms}. Since this holds
for every $P(t)<u\leq0$, we conclude that $P_{E_v}(t)\leq P(t)$
whenever $P(t)<0$.
The right-continuity of the pressure function ensures that $P_{E_v}(t)\leq P(t)$
if $P(t)=0$ for some $t$. Hence $P_{E_v}(t)\leq P(t)$ whenever $P(t)\leq0$. Since
the vertex $v$ was chosen arbitrarily, we conclude that
$\max_{v\in V}P_{E_v}(t)\leq P(t)$ whenever $P(t)\leq0$.

Thirdly, we prove that $P(t)\leq\max_{v\in V}P_{E_v}(t)$ for all $t\geq0$. To do
this, fix $t\geq0$. Let $W\subset E_A^*$ be a finite set witnessing the
irreducibility of $A$. Let $\rho:=\min\{\|\varphi_\tau'\|:\tau\in W\}$ and
$\lambda:=\max\{|\tau|:\tau\in W\}$. Then $\rho>0$ and $\lambda<\infty$.
For every $v\in V$ and every $e\in E$ there exist $\a_v(e),\b_v(e)\in W$ such
that $i(\a_v(e))=v$, $t(\b_v(e))=v$ and $\a_v(e)e\b_v(e)\in E_A^*$. Set
$\a_v(\om):=\a_v(\om_1)$ and $\b_v(\om):=\b_v(\om_{|\om|})$ for every $\om\in E_A^*$.
As previously, note that the function $\om\mapsto\a_v(\om)\om\b_v(\om)$ is at most $\lambda^2$-to-one
and that $\a_v(\om)\om\b_v(\om)\in\cup_{k=|\om|}^{2\lambda+|\om|}E_A^k$, that
is, every word $\om$ generates words $\a_v(\om)\om\b_v(\om)$ whose lengths are at most
$2\lambda+|\om|$ in the alphabet $E$ and thereby whose $v$-lengths, i.e. as a
concatenation of letters of the alphabets $E_v$, are at most $2\lambda+|\om|$.
Moreover, there is a vertex $v(\om)\in V$ such that the
word $\om$ visits $v(\om)$ at least $[|\om|/|V|]+1$, where $[\cdot]$
denotes the integer part function. This means that the $v(\om)$-length
of $\a_{v(\om)}(\om)\om\b_{v(\om)}(\om)$ is at least $[|\om|/|V|]$.
For each $v\in V$ and each $n\in\N$, let $[n/|V|]\leq k(v,n)\leq 2\lambda+n$ be such that
\[
\sum_{\tau\in E_v^{k(v,n)}}\|\varphi_\tau'\|^t
=\max_{[n/|V|]\leq k\leq 2\lambda+n} \sum_{\tau\in E_v^k}\|\varphi_\tau'\|^t.
\]
Thereafter, let $v(n)\in V$ be such that
\[
\sum_{\tau\in E_{v(n)}^{k(v(n),n)}}\|\varphi_\tau'\|^t
=\max_{v\in V} \sum_{\tau\in E_v^{k(v,n)}}\|\varphi_\tau'\|^t.
\]
Then for every $n\in\N$, we have
\begin{eqnarray*}
\sum_{\om\in E_A^n}\|\varphi_\om'\|^t
&\leq&(K\rho^{-1})^{2t}
     \sum_{\om\in E_A^n}\|\varphi_{\a_{v(\om)}(\om)\om\b_{v(\om)}(\om)}'\|^t \\
&\leq&\lambda^2(K\rho^{-1})^{2t}\sum_{v\in V}
      \sum_{\tau\in\bigcup_{k=[n/|V|]}^{2\lambda+n}E_v^k}\|\varphi_\tau'\|^t \\
&=&\lambda^2(K\rho^{-1})^{2t}\sum_{v\in V}
      \sum_{k=[n/|V|]}^{2\lambda+n}
      \sum_{\tau\in E_v^k}\|\varphi_\tau'\|^t \\
&\leq&\lambda^2(K\rho^{-1})^{2t}\sum_{v\in V}
      (2\lambda+n)\max_{[n/|V|]\leq k\leq 2\lambda+n} \sum_{\tau\in E_v^k}\|\varphi_\tau'\|^t \\
&=&\lambda^2(K\rho^{-1})^{2t}(2\lambda+n)\sum_{v\in V}\sum_{\tau\in E_v^{k(v,n)}}\|\varphi_\tau'\|^t \\
&\leq&\lambda^2(K\rho^{-1})^{2t}(2\lambda+n)|V|\max_{v\in V}\sum_{\tau\in E_v^{k(v,n)}}\|\varphi_\tau'\|^t \\
&=&\lambda^2(K\rho^{-1})^{2t}|V|(2\lambda+n)\sum_{\tau\in E_{v(n)}^{k(v(n),n)}}\|\varphi_\tau'\|^t.
\end{eqnarray*}
Since $|V|<\infty$, there exists $v\in V$ and a strictly increasing subsequence
$\{n_m\}_{m\in\N}$ of natural numbers such that $v(n_m)=v$ for all $m\in\N$.
Therefore
\begin{eqnarray*}
P(t)
&=&\lim_{m\to\infty}\frac{1}{n_m}\log\sum_{\om\in E_A^{n_m}}\|\varphi_\om'\|^t \\
&\leq&\lim_{m\to\infty}\frac{1}{n_m}\log\left(\lambda^2(K\rho^{-1})^{2t}|V|\right)
      +\lim_{m\to\infty}\frac{1}{n_m}\log(2\lambda+n_m)
      +\lim_{m\to\infty}\frac{1}{n_m}\log\hspace{-0.25cm}
       \sum_{\tau\in E_{v(n_m)}^{k(v(n_m),n_m)}}\hspace{-0.5cm}\|\varphi_\tau'\|^t \\
&=&0+0+\lim_{m\to\infty}\frac{1}{n_m}\log
       \sum_{\tau\in E_v^{k(v,n_m)}}\|\varphi_\tau'\|^t \\
&=&\lim_{m\to\infty}\frac{k(v,n_m)}{n_m}\frac{1}{k(v,n_m)}\log
       \sum_{\tau\in E_v^{k(v,n_m)}}\|\varphi_\tau'\|^t \\
&\leq&\lim_{m\to\infty}\frac{2\lambda+n_m}{n_m}\frac{1}{k(v,n_m)}\log
       \sum_{\tau\in E_v^{k(v,n_m)}}\|\varphi_\tau'\|^t \\
&=&\lim_{m\to\infty}\frac{2\lambda+n_m}{n_m}
      \cdot\lim_{m\to\infty}\frac{1}{k(v,n_m)}\log
       \sum_{\tau\in E_v^{k(v,n_m)}}\|\varphi_\tau'\|^t \\
&=&1\cdot P_{E_v}(t),
\end{eqnarray*}
where it is important to remember that
$\lim_{m\to\infty}k(v,n_m)\geq\lim_{m\to\infty}[n_m/|V|]=\infty$.
Thus, $P(t)\leq\max_{v\in V}P_{E_v}(t)$ for all $t\geq0$.

Taken together, the second and third parts allow us to deduce that
$P(t)=\max_{v\in V}P_{E_v}(t)$ when $P(t)\leq0$.
\end{proof}

The relationship between the pressures of the original and the associated systems ensures
that the limit sets of these systems have the same Hausdorff dimension.
\begin{cor}\label{hdimass}
If $\Phi$ is a CGDMS with a finitely irreducible matrix, then
$\theta_{E_v}\geq\theta$ and
$\dim_H(J_{E_v})=\dim_H(J_v)=\dim_H(J)$
for every vertex $v\in V$.
\end{cor}

\begin{proof}
Fix $v\in V$. If $t<\dim_H(J)$, then $P(t)>0$ and hence we deduce from the first part of
Theorem~\ref{pressass} that $P_{E_v}(t)>0$. If $t>\dim_H(J)$, then $P(t)<0$ and hence we
deduce from the second part of Theorem~\ref{pressass} that $P_{E_v}(t)<0$. Thus,
$\dim_H(J_{E_v})=\dim_H(J)=\dim_H(J_v)$.

Similarly, if $t<\theta$ then $P(t)=\infty>0$ and hence we deduce from the first part of
Theorem~\ref{pressass} that $P_{E_v}(t)\geq P(t)=\infty$. Thus, $t\leq\theta_{E_v}$. Since this
is true for all $t<\theta$, we conclude that $\theta\leq\theta_{E_v}$.
\end{proof}

The relationship between the pressures further indicates that the original and the
associated systems sometimes have similar natures.
Before our next corolarry, we recall several definitions from \cite{gdms}.
\begin{defn}
A CGDMS is strongly regular iff there exists $t \geq 0$ such that $0 < P(t) < \infty$.
\end{defn}
\begin{defn}
If a CGDMS $S$ is not regular,we call it irregular.
\end{defn}
\begin{defn}
A CGDMS $S$ is called critically regular if $P(\theta)=0$.
\end{defn}
\begin{defn}
A CGDMS $S$ is absolutely regular if every non-empty subsystem is regular.
\end{defn}
\begin{cor}\label{natureass}
Let $\Phi$ be a CGDMS with a finitely irreducible matrix. Then we have the following.
\begin{itemize}
\item If $\Phi$ is strongly regular, then each $\Phi_v$ may have any nature;
\item If $\Phi$ is critically regular, then $\Phi_v$ is either critically regular or irregular
      and $\theta_{E_v}=\theta$ for each $v$;
\item If $\Phi$ is irregular, so is every $\Phi_v$ and $\theta_{E_v}=\theta$ for each $v$.
\end{itemize}
\end{cor}

The relationship between the pressures also reveals that (at least) one of the
associated systems eventually has the same pressure as the original system.

\begin{cor}
There is some $v\in V$ such that $P(t)=P_{E_v}(t)$ for all $t\geq\dim_{H}(J)$.
\end{cor}

\begin{proof} Theorem~\ref{pressass} affirms that $P(t)=\max_{v\in V}P_{E_v}(t)$
for all $t\geq\dim_{H}(J)$. Since all the pressure functions $P_{E_v}(t)$, $v\in V$,
and $P(t)$ are real-analytic, there is $v\in V$ and an interval
$I\subset[\dim_{H}(J),\infty)$ such that $P(t)=P_{E_v}(t)$ for all $t\in I$.
The real-analyticity then ensures that $P(t)=P_{E_v}(t)$ for all $t\geq\dim_{H}(J)$.
\end{proof}

However, the following example shows that $P(t)=P_{E_v}(t)$ on $[\dim_H(J),\infty)$
may not hold for all $v\in V$.
\begin{example}
Let $\Phi$ be a SGDS (that is, a CGDS whose generators are all similarities)
whose set of vertices is $V=\{v,w,z\}$ and whose set of edges is $E=\{1,2,3,4\}$,
where
\begin{eqnarray*}
i(1)=v, & t(1)=w & \\
i(2)=w, & t(2)=z & \\
i(3)=z, & t(3)=v & \\
i(4)=z, & t(4)=w. &
\end{eqnarray*}
Observe that
\[
E_v=\left\{1(24)^n23:n\geq0\right\},
\]
\[
E_w=\left\{231,24\right\},
\]
and
\[
E_z=\left\{312,42\right\}.
\]
Let $r:=|\varphi_{24}'|=|\varphi_{42}'|$ and
$s:=|\varphi_{123}'|=|\varphi_{231}'|=|\varphi_{312}'|$.
Because all the generators are similarities,
\[
P_{E_v}(t)
=\log\sum_{\tau\in E_v}|\varphi_\tau'|^t
=\log\left(|\varphi_{123}'|^t\sum_{n=0}^\infty|\varphi_{24}'|^{nt}\right)
=\log\left(s^t\sum_{n=0}^\infty r^{nt}\right)
=\log\left(\frac{s^t}{1-r^t}\right),
\]
whereas
\[
P_{E_w}(t)=P_{E_z}(t)
=\log\left(|\varphi_{123}'|^t+|\varphi_{24}'|^t\right)
=\log(s^t+r^t).
\]
The Hausdorff dimension of the limit sets of the original system and the
associated systems is the unique $h>0$ such that $s^h+r^h=1$. When $t<h$,
we have $s^t+r^t>1$ and it follows that
$P_{E_v}(t)>P_{E_w}(t)=P_{E_z}(t)\geq P(t)$.
When $t>h$, we have $s^t+r^t<1$ and it ensues that
$P_{E_v}(t)<P_{E_w}(t)=P_{E_z}(t)=P(t)$. Since $P(t)$ and $P_{E_z}(t)$
are real-analytic functions which coincides on a non-empty interval,
they must coincide everywhere on their finiteness set, which is $[0,\infty)$.
We conclude that $P(t)=\log(s^t+r^t)$ for all $t\geq0$.
\end{example}

We shall now show that the inequality $P(t)\leq\min_{v\in V}P_{E_v}(t)$
when $t<\dim_H(J)$ may be strict. Indeed, there exist finite CGDSs
with (finitely) irreducible matrices whose associated systems $\Phi_v$
are all infinite. Then $P(0)<\infty=\min_{v\in V}P_{E_v}(0)$. Moreover,
$P(t)<\infty=\min_{v\in V}P_{E_v}(t)$ for all
$t\in[0,\min_{v\in V}\theta_{E_v})$.
The strict inequality
$P(t)<\min_{v\in V}P_{E_v}(t)$ extends to the right
of $\min_{v\in V}\theta_{E_v}$ in some cases.

\begin{example}
Take any SGDS which consists of
three vertices and one edge in each direction between every pair of vertices.
Such a finite system has finite pressure. However, each of its associated systems
is infinite and absolutely regular. To be more precise, let $V=\{v,w,z\}$ be the
set of vertices. Let $E=\{a,b,c,d,e,f\}$ be the set of edges, where
\begin{eqnarray*}
i(a)=v, & t(a)=w & \\
i(b)=w, & t(b)=v & \\
i(c)=w, & t(c)=z & \\
i(d)=z, & t(d)=w & \\
i(e)=v, & t(e)=z & \\
i(f)=z, & t(f)=v. &
\end{eqnarray*}
Because of the obvious symmetry, we may concentrate our efforts
on any given vertex, say $v$.
Observe that
\[
E_v=\left\{a(cd)^nb,a(cd)^ncf,e(dc)^nf,e(dc)^ndb:n\geq0\right\}.
\]
As all the generators are similarities, we obtain
\begin{eqnarray*}
P_{E_v}(t)
=\log\sum_{\tau\in E_v}|\varphi_\tau'|^t
&=&\log\left(
     \left(|\varphi_{ab}'|^t+|\varphi_{acf}'|^t+|\varphi_{ef}'|^t+|\varphi_{edb}'|^t\right)
     \sum_{n=0}^\infty|\varphi_{cd}'|^{nt}
     \right) \\
&=&\log\left(|\varphi_{ab}'|^t+|\varphi_{acf}'|^t+|\varphi_{ef}'|^t+|\varphi_{edb}'|^t\right)
   +\log\frac{1}{1-|\varphi_{cd}'|^t}.
\end{eqnarray*}
In particular, this shows that every associated system is absolutely regular, i.e.
$\theta_{E_v}=0$ for all $v\in V$.
The continuity of the pressure functions then asserts that there is some interval
$[0,L)$, with $L>0$, on which $P(t)<\min_{v\in V}P_{E_v}(t)$ for every $t<L$.
\end{example}

All of the above examples show that Theorem~\ref{pressass} is the best general
result one can achieve.

\begin{rem}
Section~4.7 in~\cite{gdms} contains some inaccuracies. First, the system generated
by a strictly Markov system is strictly Markov, and thus not an iterated function system,
as claimed. Moreover, the proof of Theorem~4.7.4 in~\cite{gdms} contains a mistake.
Indeed, the correct argument is:
By Proposition~\ref{jivjv}, we have $\dim_H(J_{E_v})\leq\dim_H(J_v)
=\dim_H(J)$ for every vertex $v\in V$. Since $\Phi_v$ and $\Phi$ are irreducible,
Theorem~4.2 in~\cite{AGGDMS} shows that to prove that
$\dim_H(J)\leq\dim_H(J_{E_v})$ it suffices to demonstrate that
$P(t)\leq\max\{P_{E_v}(t),0\}$ for every $t\geq0$ and every $v\in V$. To do this,
fix $v\in V$ and $t\geq0$. Let $W\subset E_A^q$ be a set witnessing the finite
primitivity of $A$. Then for every $e\in E$ there exist $\a(e),\b(e)\in W$ such
that $i(\a(e))=v$, $t(\b(e))=v$ and $\a(e)e\b(e)\in E_A^*$. Set
$\a(\om):=\a(\om_1)$ and $\b(\om):=\b(\om_{|\om|})$ for every $\om\in E_A^*$.
Let $u>\max\{P_{E_v}(t),0\}$.
Then
\begin{eqnarray*}
\sum_{\om\in E_A^*}\|\varphi_\om'\|^t e^{-u|\om|}
&\leq&
     (K\rho)^{2t}
     \sum_{\om\in E_A^*}\|\varphi_{\a(\om)\om\b(\om)}'\|^t e^{-u|\om|} \\
&=&
     (K\rho)^{2t}e^{2qu}
     \sum_{\om\in E_A^*}\|\varphi_{\a(\om)\om\b(\om)}'\|^t e^{-u|\a(\om)\om\b(\om)|} \\
&\leq&
   (K\rho)^{2t}e^{2qu}
     \sum_{\tau\in E_v^*}\|\varphi_\tau'\|^t e^{-u|\tau|} \\
&\leq&
     (K\rho)^{2t}e^{2qu}
     \sum_{\tau\in E_v^*}\|\varphi_\tau'\|^t e^{-u|\tau|_v} \\
&<&\infty,
\end{eqnarray*}
where the second inequality sign follows from the facts that
$\a(\om)\om\b(\om)\in E_v^*$ and the function
$\om\mapsto\a(\om)\om\b(\om)$ is one-to-one; the third inequality sign follows
from the facts that $u>0$ and $|\tau|_v\leq|\tau|$; the last inequality follows
from Theorem~4.2 in~\cite{AGGDMS} since $u>P_{E_v}(t)$. The resulting inequality
$\sum_{\om\in E_A^*}\|\varphi_\om'\|^t e^{-u|\om|}<\infty$ implies that $u>P(t)$
according to Theorem~4.2 in~\cite{AGGDMS}.
Since this is true for every $u>\max\{P_{E_v}(t),0\}$, we deduce that
$\max\{P_{E_v}(t),0\}\geq P(t)$. This implies $\dim_H(J_{E_v})\geq\dim_H(J)$.

Though this argument confirms the equality of the Hausdorff dimensions of the
limit sets of the original and its associated systems, it does not provide
as strong information about their pressures as in Theorem~\ref{pressass}.
\end{rem}



\begin{thebibliography}{99}

\bibitem{DK} Dajani, K. and Kraaikamp, C.: Ergodic Theory of Numbers. {\em Carus Mathematical Monographs} {\bf 29}, The Mathematical Association of America, Washington DC, 2002.

\bibitem{Fal}
K. Falconer.
\newblock {\em Fractal Geometry.} \newblock John Wiley, New York, 1990.


\bibitem{AGthesis} A. Ghenciu, {\it Dimension Spectrum and Graph Directed
Markov Systems}, Ph.D. thesis, University of North Texas, 2006.

\bibitem{AGpara} A. Ghenciu, {\it Parabolic iterated function systems with
applications to the backward continued fractions}, {\it Far East Journal of
Dynamical Systems}, {\bf 9} (1) (2007) 75--91.

\bibitem{G} A. Ghenciu, {\it Gauss-like continued fraction systems and their
dimension spectrum}, {\it Real Analysis Exchange}, {\bf ???}.

\bibitem{AGGDMS} A. Ghenciu and R.D. Mauldin,
{\it Conformal graph directed Markov systems}, preprint, arXiv:0711.1182v1.

\bibitem{hur89} A. Hurwitz, {\it \"Uber eine besondere art der kettenbruchentwicklung reeller gr\"ossen}, Acta. Math., {\bf 12} (1889),
367--404.

\bibitem{jag85} H. Jager, {\it Metrical results for the nearest integer continued fraction}, Indag. Math., {\bf 88} no 4,  (1985), 417--427.

\bibitem{jagkra} H. Jager and C. Kraaikamp, {\it On the approximation by continued fractions}, Indag. Math., {\bf 51} no 3, (1989), 289--307.

\bibitem{KZ} M. Kesseb\"ohmer and S. Zhu, {\it Dimension sets for infinite
IFSs: The Texan Conjecture}, Journal of Number Theory, {\bf 116} (2006),
230--246.

\bibitem{Kra} C. Kraaikamp, {\em A new class of continued fraction expansions}, Acta Arithmetica {\bf 58} no. 1, (1991), 1--39.

\bibitem{MU} R. D. Mauldin and M. Urba\'nski, {\it Dimensions
and measures in infinite iterated function systems}, Proc.
London Math. Soc. {\bf 73} (3) (1996), no. 1, 105--154.

\bibitem{mutr} R. D. Mauldin and M. Urba\'nski, {\em Conformal
iterated function systems with applications to the geometry
of continued fractions},  Trans. Amer. Math. Soc. {\bf 351}
(1999), 4995--5025.

\bibitem{gdms}
R.D. Mauldin and M.Urba\'nski, \emph{Graph Directed Markov Systems},
Cambridge Tracts in Mathematics {\bf 148},
Cambridge University Press, Cambridge, 2003.

\bibitem{min73} B. Minnegerode, {\it \"Uber eine neue methode, die pellsche gleichung aufzul\"osen}, Nachr. G\"ottingen, (1873).


\bibitem{MW} R. D. Mauldin and S. C. Williams, {\it Hausdorff
dimension in graph directed constructions}, Trans. Amer. Math. Soc.
{\bf 309} (2) (1988), 811--829.

\bibitem{rock80} A Rockett, {\it The metrical theory of continued fractions to the nearest integer}, Acta. Arith., {\bf 38} (1980), 97--103.

\bibitem{R}
M. Roy,
\emph{A new variation of Bowen's formula for graph directed Markov systems},
to appear in {\it Discrete and Continuous Dynamical Systems}.
\bibitem{sA}
O. Sarig,
\emph{Thermodynamic formalism for countable Markov shifts},
Erg. Th. Dynam. Sys. {\bf 19} (1999), 1565--1593.


\end{thebibliography}
\end{document}